\documentclass[12pt,a4paper]{article}
\usepackage{hyperref}
\usepackage{mathtools}
\usepackage[nottoc,numbib]{tocbibind}
\usepackage[english]{babel}
\usepackage[utf8x]{inputenc}
\usepackage{amsfonts,longtable,amssymb,amsmath,latexsym,dsfont,mathabx}
\usepackage{wrapfig}
\usepackage[title]{appendix}
\usepackage{amsthm, amsmath, amssymb, amsfonts, amscd, amscd, geometry}
\newsavebox{\ssa}
\usepackage{bbm}
\usepackage{float}
\usepackage{tikz}

\DeclareMathOperator{\disc}{disc}
\DeclareMathOperator{\Rep}{Rep}

\newcommand{\floor}[1]{\left\lfloor#1\right\rfloor}
\newcommand{\abs}[1]{\left|#1\right|}

\newcommand{\bigbinom}[2]{\left(\genfrac{}{}{0pt}{}{#1}{#2}\right)}
\makeatletter
\newcounter{savesection}
\newcounter{apdxsection}
\renewcommand\appendix{\par
  \setcounter{savesection}{\value{section}}%
  \setcounter{section}{\value{apdxsection}}%
  \setcounter{subsection}{0}%
  \gdef\thesection{\@Alph\c@section}}
\newcommand\unappendix{\par
  \setcounter{apdxsection}{\value{section}}%
  \setcounter{section}{\value{savesection}}%
  \setcounter{subsection}{0}%
  \gdef\thesection{\@arabic\c@section}}
\makeatother

\newtheorem{theorem}{Theorem}[section]
\newtheorem{corollary}[theorem]{Corollary}

\newtheorem{lemma}[theorem]{Lemma}
\newtheorem{prop}[theorem]{Proposition}

\theoremstyle{definition}
\newtheorem{definition}[theorem]{Definition}
\newtheorem{remark}[theorem]{Remark}
\usepackage[colorinlistoftodos]{todonotes}
\usepackage{verbatim}

\usepackage[capitalize]{cleveref}
	\crefname{claim}{Claim}{Claims}
	\Crefname{claim}{Claim}{Claims}
	\crefname{app-corollary}{Corollary}{Corollaries}
	\Crefname{app-corollary}{Corollary}{Corollaries}
	\crefname{app-definition}{Definition}{Definitions}
	\Crefname{app-definition}{Definition}{Definitions}
	\crefname{figure}{Figure}{Figures}
	\Crefname{figure}{Figure}{Figures}
	\crefname{lemma}{Lemma}{Lemmata}
	\Crefname{lemma}{Lemma}{Lemmata}
	\crefname{app-lemma}{Lemma}{Lemmata}
	\Crefname{app-lemma}{Lemma}{Lemmata}
	\crefname{app-proposition}{Proposition}{Proposition}
	\Crefname{app-proposition}{Proposition}{Proposition}
	\crefname{app-theorem}{Theorem}{Theorems}
	\Crefname{app-theorem}{Theorem}{Theorems}
	
\usepackage{amsmath}
\usepackage{color,colortbl}
\usepackage{subcaption}

\def \C {\mathbb{C}}

\def \F {\mathbb{F}}
\def \N {\mathbb{N}}

\def \Q {\mathbb{Q}}
\def \Z {\mathbb{Z}}

\def \KK {\mathcal{K}}

\def \OO {\mathcal{O}}
\def \PP {\mathcal{P}}
\def \QQ {\mathcal{Q}}

\def \Ind {\textrm{Ind}}
\def \Dim {\textrm{Dim}}

\def \a {\alpha}
\def \b {\beta}
\def \g {\gamma}

\def \d {\delta}
\def \k {\kappa}

\def \m {\mu}

\def \pp {\mathfrak{p}}
\def \qq {\mathfrak{q}}

\title{Sums of binomial coefficients evaluated at $\a \in \overline{\Q},$ and applications}

\author{
  Daniil Kalinov, Andrei Mandelshtam}

\begin{document}
\maketitle

\begin{abstract}
The additive monoid $R_+(x)$ is defined as the set of all nonnegative integer linear combinations of binomial coefficients $\binom{x}{n}$ for $n \in \mathbb Z_+$. 
This paper is concerned with the inquiry into the structure of $R_+(\a)$ 
for complex numbers $\a.$ 
Particularly interesting is the case of algebraic $\a$ which are not non-negative integers. This question is motivated by the study of functors between Deligne categories $\Rep(S_t)$ (and also $\Rep(\textrm{GL}_t)$) for $t \in \mathbb C\backslash \mathbb Z_+$. We prove that this object is a ring if and only if $\a$ is an algebraic number that is not a nonnegative integer. Furthermore, we show that all algebraic integers generated by $\a,$ i.e. all elements of $\OO_{\Q(\a)},$ are also contained in this ring. We also give two explicit representations of $R_+(\a)$ for both algebraic integers and general algebraic numbers $\a.$ 
One is in terms of inequalities for the valuations with respect to certain prime ideals and the other is in terms of explicitly constructed generators. 
We show how these results work in the context of the study of symmetric monoidal functors between Deligne categories in positive characteristic.
Moreover, this leads to a particularly simple description of $R_+(\a)$ for both quadratic algebraic numbers and roots of unity.
\end{abstract}

\tableofcontents

\section{Introduction}

In the ring $\Q[x]$ of polynomials with coefficients in the rational numbers, it is interesting to consider the subring of all \textit{integer-valued} polynomials, i.e. polynomials $p(x)$ such that $p(n)$ is an integer for every integer $n.$ This ring is known as the most natural and simple example of a non-Noetherian ring. One may wonder whether this is not just the set of all polynomials with integer coefficients. However, e.g. the polynomial $\frac{x^2+x}{2}$ is integer-valued. It turns out that these are exactly the polynomials with integer coefficients in the basis of binomial coefficients $\binom{x}{n}\in\Q[x].$ In this paper, we examine the set $R_+(x)$ of polynomials which have nonnegative integer coefficients in this basis. More precisely, we study the set of values of these polynomials at a fixed number $\a.$ More precisely, the set $R_+(\a),\a\in\C,$ is defined as a monoid\footnote{Recall that a monoid is a group that doesn't need to have inverses. In the example at hand the group operation is addition.}  spanned by the non-negative integer linear combinations of binomial coefficients $\binom{\a}{k}$ for $k \in \Z_{\geq0}.$ It turns out that this monoid is closed under multiplication and, hence, gives us a semiring\footnote{Recall that a semiring is a ring that doesn't need to have additive inverses.}.


The ring of integer-valued polynomials often appears in representation theory, especially when working with representation theory in positive characteristic. It also naturally shows up in the domain of  Deligne categories. One of the ways in which it makes it appearance was discussed in the paper  \cite{classify-symmetric-monoidal-functors} and this paper also provides the main reason for the study of $R_+(\a)$.

Deligne categories were first introduced by Deligne in \cite{deligne2007categorie}. They serve as an interpolation of regular categories of finite-dimensional representations of symmetric groups $S_n$, which we denote by \textbf{Rep}$(S_n)$ (note the bold font), to complex values of $n.$ In this sense the group $S_t$ for $t\in\C\backslash\Z_{\geq0}$ does not exist, but the tensor category of representations \text{Rep}$(S_t)$ can still be defined. One of the reasons why this category is an interpolation of categories \textbf{Rep}$(S_n)$ is that for $n\in\Z_{\geq0}$ the category \textbf{Rep}$(S_n)$ turns out to be a quotient of \text{Rep}$(S_n)$ by the ideal of nilpotent morphisms. It is known that this category also admits no fiber functor, i.e. we cannot think of this object as having an underlying vector space structure. Nevertheless one can comfortably work with this category. Deligne also defined similar categories in the case of general linear, symplectic and special orthogonal groups.

With the papers of Etingof \cite{etingof2014representation,etingof2016representation} the new frontier of {\it representation theory in complex rank} was opened. There, he proposed a series of problems concerned with generalizing statements from the usual finite rank representation theory to the Deligne category setting. The paper \cite{classify-symmetric-monoidal-functors} was concerned with one such problem. As a byproduct of its solution, the authors of that paper were able to prove the following result:
\begin{theorem}[Harman and Kalinov \cite{classify-symmetric-monoidal-functors}] \label{symm-mon-funct-prop}
Symmetric monoidal functors between the Deligne categories \text{Rep}$(S_t) \to$\text{Rep}$(S_{t'})$ with $t,t'\notin \Z_{\ge 0}$ exist if and only if $t \in R_+(t')$. 
\end{theorem}

Also note that a similar result holds for Deligne categories $\Rep(GL_t)$:
\begin{theorem} \label{symmetric-monoidal-functor-thm}
Symmetric monoidal functors between the Deligne categories $\Rep(GL_t) \to \Rep(GL_{t'})$ with algebraic $t,t' \notin \Z_{\geq0}$ exist if and only if $t \in R_+(t')$.
\end{theorem}
\begin{proof}
From the universal property of $\Rep(GL_t)$ we know that such a functor exists iff there is an object of dimension $t$ in $\Rep(GL_{t'})$. Thus we only need to understand the structure of the set of possible dimensions of objects in $\Rep(GL_{t'})$.

Now note that all simple objects of $\Rep(GL_{t'})$ have dimension given by the value of an integer-valued polynomial on $t'.$ In other words, the dimensions of all simple objects, and hence of all objects, are in $R(t')$ (where $R(s)$ is the ring spanned by all integer combinations of binomial coefficients in $s$). By Section \ref{Ring-Section}, for an algebraic $t'$ we have the equality $R_+(t')= R(t').$ It follows that the dimensions of all objects of $\Rep(GL_{t'})$ lie in $R_+(t')$. 

But now also note that the interpolations of the exterior powers of the fundamental representation have dimensions given by the binomial coefficients of $t'$ themselves. Hence, for any $q \in R_+(t')$ there is an object in $\Rep(GL_{t'})$ of dimension $q.$ Therefore, the set of dimensions of objects of $\Rep(GL_{t'})$ is exactly $R_+(t')$ and the result follows.
\end{proof}
Thus, to understand whether a symmetric monoidal functor exists between such Deligne categories, it is helpful to determine the explicit form of $R_+(\a)$. Since, in the finite rank, symmetric monoidal functors between categories of representations of symmetric groups are in one-to-one correspondence with the homomorphisms between symmetric groups themselves, it follows that by studying the structure of $R_+(\a)$ we are essentially studying the interpolations of homomorphisms between symmetric groups to non-integer rank.

Previously, $R_+(\a)$ was described explicitly for $\a\in\mathbb{Q}$ in \cite{classify-symmetric-monoidal-functors}, where this object was introduced. After that, Hu \cite{sqrt-2-case} described explicitly the structure of $R_+(\sqrt{2})$ and made advances on the structure of $R_+(\sqrt{n}).$ In the current paper, we provide the full solution of this problem. The question is interesting only for algebraic numbers $\a,$ since $R_+(t)$ are isomorphic to each other for all transcendental $t,$ and this set is not a ring. One of our main results is the following.
\begin{theorem}\label{intro-its-a-ring}
The semiring $R_+(\a)$ is a ring if and only if $\a$ is an algebraic number that is not a nonnegative integer.
\end{theorem}
This is proved as Corollary \ref{its-a-ring} of Theorem \ref{multiply-to-get-pos-coeffs}. Moreover, it turns out that this ring has an interesting structure that can be exactly characterized for any algebraic number $\a.$ 
This is done in our second main result. Namely, for each $\a$ and each prime integer $p$ we define a set $\PP_p'$ that consists of prime ideals $\pp$ of $\OO_{\Q(\a)}$ over $(p)$ such that the ramification index of $\pp$ is $1,$ $|\OO_{\Q(\a)}/\pp| = p,$ and $v_\pp(\a)\geq0$ (see Definition \ref{P_p-updated}). Then: 
\begin{theorem}\label{exact-description-for-all-alg-nums-intro}
For any algebraic number $\a,$ $R_+(\a)$ is the set of all elements $\b\in\Q(\a)$ such that $v_\pp(\b)\geq0$ for all $\pp\in\PP_p'$ for all primes $p.$
\end{theorem}
We prove this result as Theorem \ref{exact-description-for-all-alg-nums-valuation} in the main text. Later, in Section \ref{p-adic-functor-section}, we slightly reformulate this theorem and show why it is in fact a natural result by showing how it is tied to a certain statement about the connection between symmetric monoidal functors between Deligne categories over $\C$ and Deligne categories in positive characteristic (see Remark \ref{remark-p-char-deligne}).

This description is in terms of ideals of the ring of integers of the field generated by $\a.$ However, sometimes it is also valuable to have a more explicit characterization of the ring in terms of the set of its generators. Such a description is given in our third main result. For each $\a,$ we define certain explicit elements $\a_p\in\Q(\a)$ for primes $p$ (see Definition \ref{a_p}). Then:
\begin{theorem}\label{final-description-all-nums-intro}
For $\a$ an algebraic number with integral part $\b,$ the ring $$
R_+(\a)=\OO_{\Q(\a)}\left[\a,\b_2,\b_3,\b_5,\dots\right].
$$
In particular, if $\a$ is an algebraic integer, then $$
R_+(\a)=\OO_{\Q(\a)}\left[\a_2,\a_3,\a_5,\dots\right].
$$
\end{theorem}
We restate and prove this in Theorem \ref{final-description-all-nums}.
Furthermore, we provide even simpler descriptions of $R_+(\a)$ for quadratic algebraic numbers and roots of unity $\a,$ see Section \ref{Examples-Section}, allowing us to easily answer the question of whether $\b\in R_+(\a)$ for an algebraic number $\b.$

Another important result is the following lemma, proved by Harman and Kalinov \cite{classify-symmetric-monoidal-functors} using methods of category theory:
\begin{lemma}[Harman and Kalinov \cite{classify-symmetric-monoidal-functors}]\label{binomial-coeffs-of-binomial-coeffs-1}
If $\b\in R_+(\a),$ then $R_+(\b)\subset R_+(\a).$
\end{lemma}
Here, we prove this result without the use of category theory, using only combinatorial techniques, see Lemma \ref{binomial-coeffs-of-binomial-coeffs}.

While the results of Section \ref{Examples-Section} are proved using Theorem \ref{exact-description-for-all-alg-nums}, these proofs may be difficult to visualize and understand. The Appendix contains alternative, explicit proofs that demonstrate a different, constructive approach.

Section \ref{Background-Section} provides background information about the problem. In Section \ref{x-in-R+y-then-R+x-Section} we prove Lemma \ref{binomial-coeffs-of-binomial-coeffs-1} combinatorially. In Section \ref{Ring-Section} we prove Theorem \ref{intro-its-a-ring}. In Section \ref{General-Section} we prove several versions of Theorem \ref{exact-description-for-all-alg-nums-intro}. In Section \ref{Specifics-Section} we construct explicit generators for $R_+(\a)$ and prove Theorem \ref{final-description-all-nums-intro}. In Section \ref{Examples-Section} we describe and prove the structure of $R_+(\a)$ for several important classes of numbers $\a.$ Finally, alternative, constructive proofs of results from Section \ref{Examples-Section} are located in the Appendix.

\textbf{Acknowledgments:} The authors thank RSI for the opportunity to work on this problem. They thank Pavel Etingof for supervising this project and Alexander Petrov for useful discussion on the number-theoretic concepts involved in this paper. And they also thank Pavel Etingof, Tanya Khovanova and Alexander Petrov for proofreading the paper. 
\section{Preliminary information}\label{Background-Section}
\begin{definition}
For any (not necessarily algebraic) number $\a,$ the set $R_+(\a)$ is defined as the monoid generated by the binomial coefficients of $\a,$ i.e. $$
R_+(\a)=\Z_{\geq0}\left<1,\a,\binom{\a}{2},\dots,\binom{\a}{k},\dots\right>,
$$
where $\binom{\a}{k}=\frac{\a(\a-1)\cdots(\a-k+1)}{k!}.$ Explicitly, $t\in R_+(\a)$ if it can be written in the form $$
t=a_0+a_1\a+a_2\binom{\a}{2}+\dots+a_n\binom{\a}{n}
$$
for some $n,$ where the $a_i$ are nonnegative integers.
\end{definition}
We extend this definition to $\a$ an indeterminate. Specifically, we define $R_+(x)$ as the subset of $\Q[x]$ generated as an additive monoid by the binomial coefficients (defined above) of $x.$ Note that in this case $R_+(x)$ is a subset of the ring $R(x)$ of polynomials that are integer-valued on integer inputs.
\subsection{Integer Case}
For any nonnegative integer $n,$ we have $\binom{n}{k}=0$ for $k>n$ so, because $\binom{n}{k}$ is nonnegative for all $k\geq0,$ we have $$
R_+(n)=\Z_{\geq0}\left<1,n,\binom{n}{2},\dots,\binom{n}{n}\right>=\Z_{\geq0}.
$$
On the other hand, if $n<0,$ we have $-n-1+n=-1\in R_+(n),$ so the element $-1,$ along with $1,$ generates all of $\Z.$ Since for any nonnegative integer $k$, $\binom{n}{k}$ is an integer\footnote{Indeed, for $n<0$ we note that $\binom{n}{k}=(-1)^k\binom{-n+k-1}{k},$ which is an integer.}, it follows that $R_+(n)=\Z.$
\subsection{Semiring and Ring Structure}
Here follow some helpful lemmas concerning the general structure of $R_+(\a),$ proved by Harman and Kalinov \cite{classify-symmetric-monoidal-functors} (also see Hu \cite{sqrt-2-case}), but we provide slightly different proofs:
\begin{lemma}\label{semiring}
$R_+(\a)$ is a semiring, i.e. it is closed under addition and multiplication, and has an additive and multiplicative identity.
\end{lemma}
\begin{proof}
$R_+(\a)$ is by definition closed under addition. Note that for closure under multiplication, it suffices to show that $\binom{\a}{k}\binom{\a}{l}\in R_+(\a)$ for any $k,l.$ But this follows from the simple combinatorial identity $$
\binom{\a}{k}\binom{\a}{l}=\sum_{i=0}^l\binom{k}{i}\binom{k+l-i}{k}\binom{\a}{k+l-i}.
$$
This identity can be proven by noting that a subset of $k$ elements and a subset of $l$ elements chosen from a set of $\a$ elements intersect in some $i$ elements, and by considering the number of possibilities for each $i.$ This identity is true for all positive integers $\a,$ but it is a polynomial identity in $\a,$ so it must work for all $\a,$ not necessarily positive integers. The identity immediately implies that $\binom{\a}{k}\binom{\a}{l}$ is a positive linear combination of binomial coefficients of $\a,$ as desired. Finally, $0,1\in R_+(\a),$ and they are the additive and multiplicative identities, respectively, completing the proof.
\end{proof}
From here we obtain a simple characterization of when $R_+(\a)$ is a ring:
\begin{lemma}\label{negative-one-ring}
$R_+(\a)$ is a ring if and only if $-1\in R_+(\a).$
\end{lemma}
\begin{proof}
The forward direction follows from the fact that $-1$ is the additive inverse of $1\in R_+(\a).$ For the backwards direction, it is sufficient to prove that additive inverses exist. But if $a\in R_+(\a),$ then $(-1)(a)=-a\in R_+(\a),$ as desired.
\end{proof}
\subsection{Rational Case}
A more subtle case than $\a$ an integer occurs for $\a$ a rational number. This case also shows some of the behaviors which one can also see in the general case. For this reason we discuss it separately below, before we treat the general case in Section \ref{General-Section}.

For example, say $\a=\frac{1}{2}.$ Then we have $$
\binom{\a}{2}=\frac{(1/2)(-1/2)}{2}=-\frac{1}{8}.
$$
We multiply by $8$ to see that $-1\in R_+(\a).$ By Lemma \ref{negative-one-ring} we see that this implies that $R_+\left(\frac{1}{2}\right)$ is a ring. In particular, it follows that $\Z\left[\frac{1}{2}\right],$ the set of all rational numbers with denominators that are powers of $2,$ is contained in $R_+(\a).$ On the other hand, we have \begin{align*}
    \binom{\a}{k}=\frac{\a(\a-1)\cdots(\a-k+1)}{k!}=\frac{1(1-2)\cdots(1-2(k-1))}{k!2^k}&=\frac{(2k-2)!(-1)^{k-1}}{(k-1)!k!2^k}\\
    &=\frac{C_{k-1}(-1)^{k-1}}{2^k},
\end{align*}
where $C_n,$ the $n$th Catalan number, is an integer. Hence, the number $\binom{\a}{k}\in\Z\left[\frac{1}{2}\right].$ Also, the binomial coefficients of $\a$ generate $R_+\left(\frac{1}{2}\right),$ implying that $R_+\left(\frac{1}{2}\right)=\Z\left[\frac{1}{2}\right].$ 

For general rational numbers $\a,$ however, we need some more notation to prove that $R_+(\a)=\Z[\a]$.

An important function on the integers, that is used throughout this paper, is the $p$-adic valuation $v_p(x),$ defined informally as the exponent of $p$ in the prime factorization of $x.$ Formally, it is defined for a prime $p$ recursively by $v_p(x)=0$ if $p$ and $x$ are relatively prime and $v_p(x)=1+v_p(x/p)$ otherwise. It is not difficult to see that $v_p(xy)=v_p(x)+v_p(y)$ and Legendre's formula shows that $v_p(k!)=\sum\limits_{n=1}^\infty\lfloor\frac{k}{p^n}\rfloor.$ Also, it is not hard to show that $v_p(a+b)\geq\min(v_p(a),v_p(b)).$
\begin{lemma}[Harman and Kalinov \cite{classify-symmetric-monoidal-functors}]\label{rational-case}
If $\a=\frac{r}{q}\in\Q$ is a rational number with $(r,q)=1$ $(r$ and $q$ are relatively prime) and $q\neq1,$ then $R_+(\a)=\Z[\a]=\Z\left[\frac{1}{q}\right].$
\end{lemma}
We include the proof of this lemma because it provides a certain insight into the general case, proved in Proposition \ref{exact-ideal-representation}.
\begin{proof}
First, assume that $\a$ is positive. Then say $(\a$ is not an integer) $n<\a<n+1$ for some integer $n.$ We then have that $$\binom{\a}{n+2}=\frac{\a(\a-1)\dots(\a-n)(\a-n-1)}{(n+2)!}=\frac{a}{b}<0,$$ with $a<0$ and $b>0.$ We multiply by $b$ to get that $b\binom{\a}{n+2}=a<0\in R_+(\a).$ But then $-1-a\geq0\in R_+(\a),$ so we add them to get that $-1\in R_+(\a).$ If $\a=\frac{r}{q}<0,$ then $q\a=r<0$ so $-r-1+r=-1\in R_+(\a).$ In any case, $-1\in R_+(\a),$ thus, by Lemma \ref{negative-one-ring}, $R_+(\a)$ is a ring so $\Z[\a]\subseteq R_+(\a).$ Thus, it suffices to show that $R_+(\a)\subseteq\Z[\a].$

We show that, for any $k,$ $\binom{\a}{k}\in\Z[\a].$ Indeed, because the binomial coefficients of $\a$ generate $R_+(\a),$ it would follow that $R_+(\a)\subseteq\Z[\a],$ as desired. Now, we have \begin{align*}
\binom{\a}{k}=\frac{\a(\a-1)\dots(\a-k+1)}{k!}&=\frac{(r/q)(r/q-1)\dots(r/q-k+1)}{k!}\\
&=\frac{r(r-q)\dots(r-q(k-1))}{q^kk!}.
\end{align*}
For any prime $p$ not dividing $q,$ $q$ must be relatively prime to $p^n$ for every positive integer $n.$ Hence, $q$ has an inverse modulo $p^n,$ say $qt_n\equiv1\pmod{p^n}.$ Then $qt_nr\equiv r\pmod{p^n},$ so $r-q(t_nr)$ is divisible by $p^n.$ Note that this argument holds for any number equivalent to $t_nr$ modulo $p^n,$ in particular there are at least $\lfloor\frac{k}{p^n}\rfloor$ numbers $i$ from $0$ to $k-1$ for which $r-qi$ is divisible by $p^n.$

But we see that in the product $r(r-q)\dots(r-q(k-1)),$ at least $\lfloor\frac{k}{p}\rfloor$ factors are divisible by $p,$ so they contribute one factor of $p$ to the product. At least $\lfloor\frac{k}{p^2}\rfloor$ factors are divisible by $p^2,$ so they contribute an extra factor of $p$ to the product. We continue, to get that $$
v_p(r(r-q)\dots(r-q(k-1)))\geq\sum\limits_{n=1}^\infty\floor{\frac{k}{p^n}}=v_p(k!).
$$
Thus, when the expression $\binom{\a}{k}$ is put in simplest terms, the denominator has no factors of $p.$ But $p$ was chosen as an arbitrary prime not dividing $q,$ so all primes dividing the denominator of the simplified form of $\binom{\a}{k}$ also divide $q,$ hence $\binom{\a}{k}\in\Z\left[\frac{1}{q}\right]=\Z[\a],$ as desired.
\end{proof}
Lemma \ref{rational-case} provides the full characterization of $R_+(\a)$ for rational numbers $\a.$ 

\subsection{Finite Difference Operators}
We introduce a few definitions and results that are valuable for the calculation of the coefficients of the binomial expansion of a general polynomial. They will be used in Sections \ref{x-in-R+y-then-R+x-Section} and \ref{Ring-Section}. All of the results of this section are well-known and can be found in Jordan and Jord\'an \cite{finite-differences}, but we include the proofs for convenience. 
\begin{definition}\label{difference-operator-definition}
For a polynomial $p(x),$ define the finite difference $\Delta p(x)=p(x+1)-p(x).$ For $n=0,$ define $\Delta^0p(x)=p(x),$ and for any nonnegative integer $n,$ define the iterated finite difference $$
\Delta^{n+1}p(x)=\Delta(\Delta^np(x))=\Delta^np(x+1)-\Delta^np(x).
$$
\end{definition}

Note that if $p(x)$ is of degree $n,$ then $\Delta p(x)=p(x+1)-p(x)$ is of degree $n-1.$ To see this, note that the highest-degree terms in the monomial expansions of $p(x)$ and $p(x+1)$ are equal. 
An important property of the finite difference operator $\Delta$ is that it behaves nicely when applied to binomial coefficients. Specifically, $$
\Delta\binom{x+a}{k}=\binom{x+a+1}{k}-\binom{x+a}{k}=\binom{x+a}{k-1}.
$$
From here, we see that $\Delta^i\binom{x+a}{k}=\binom{x+a}{k-i}.$
\begin{lemma}[Cahen and Chabert \cite{int-value-polys}]\label{expansion-into-binomial-coefficients}
For any polynomial $p(x)$ of degree $n,$ its expansion as a linear combination of binomial coefficients is $p(x)=\sum\limits_{i=0}^n(\Delta^ip)(0)\binom{x}{i}.$
\end{lemma}
\begin{proof}
Since each $\binom{x}{i}$ is a polynomial of degree $i,$ the collection of all binomial coefficients forms a basis for $\Q[x]$ as a vector space over $\Q.$ In particular, say that $p(x)=\sum\limits_{i=0}^na_i\binom{x}{i}.$ We see that $$
\Delta^kp(x)=\sum\limits_{i=k}^na_i\binom{x}{i-k}=\sum\limits_{i=0}^{n-k}a_{k+i}\binom{x}{i}.
$$
Then $$
(\Delta^kp)(0)=\sum\limits_{i=0}^{n-k}a_{k+i}\binom{0}{i}=a_k\binom{0}{0}=a_k,
$$
as desired.
\end{proof}
We can also visualize $\Delta^nf(x)$ in another way.
\begin{lemma}[Jordan and Jord\'an \cite{finite-differences}]\label{iterated-finite-difference}
For any nonnegative integer $n,$ $\Delta^nf(x)=\sum\limits_{i=0}^n(-1)^i\binom{n}{i}f(x+n-i).$
\end{lemma}
\begin{proof}
We prove this by induction on $n.$ The base case, $n=0,$ is trivial - the sum is just $f(x).$ Now, suppose the formula is true for $n.$ Then we have \begin{align*}
&\Delta^{n+1}f(x)=\Delta^nf(x+1)-\Delta^nf(x)\\
&=\sum_{i=0}^n(-1)^i\binom{n}{i}f(x+1+n-i)-\sum_{i=0}^n(-1)^i\binom{n}{i}f(x+n-i)\\
&=f(x+n+1)-(-1)^nf(x)+\sum_{i=0}^{n-1}(-1)^{i+1}\left(\binom{n}{i+1}f(x+n-i)+\binom{n}{i}f(x+n-i)\right)\\
&=f(x+n+1)+(-1)^{n+1}f(x)+\sum_{i=0}^{n-1}(-1)^{i+1}\binom{n+1}{i+1}f(x+n-i)\\
&=\sum_{i=0}^{n+1}(-1)^i\binom{n+1}{i}f(x+n+1-i),
\end{align*}
completing the inductive step.
\end{proof}
Just like for the derivative, there is a version of the product rule for $\Delta.$ Specifically, we have \begin{align*}
    \Delta(p(x)q(x))&=p(x+1)q(x+1)-p(x)q(x)\\
    &=p(x+1)q(x+1)-p(x+1)q(x)+p(x+1)q(x)-p(x)q(x)\\
    &=p(x+1)\Delta q(x)+q(x)\Delta p(x).
\end{align*}
This generalizes for $\Delta^k$:
\begin{lemma}[Leibniz Rule]\label{product-rule-for-difference-quotient}
$\Delta^k(p(x)q(x))=\sum\limits_{i=0}^k\binom{k}{i}\Delta^{k-i}p(x+i)\Delta^iq(x).$
\end{lemma}
\begin{proof}
We prove this by induction. The base case $k=1$ was shown above in the calculation of $\Delta(p(x)q(x)).$ Now, suppose the formula holds true for $k.$ Then we have \begin{align*}
&\Delta^{k+1}(p(x)q(x))=\Delta(\Delta^k(p(x)q(x)))=\Delta\left(\sum_{i=0}^k\binom{k}{i}\Delta^{k-i}p(x+i)\Delta^iq(x)\right)\\
&=\sum_{i=0}^k\left(\binom{k}{i}\Delta^{k-i}p(x+i+1)\Delta^{i+1}q(x)+\binom{k}{i}\Delta^{k-i+1}p(x+i)\Delta^iq(x)\right)
\end{align*}
\begin{align*}
&=\sum_{j=1}^{k+1}\binom{k}{j-1}\Delta^{k-j+1}p(x+j)\Delta^jq(x)+\sum_{i=0}^k\binom{k}{i}\Delta^{k-i+1}p(x+i)\Delta^iq(x)\\
&=p(x+k+1)\Delta^{k+1}q(x)+q(x)\Delta^{k+1}p(x)+\sum_{i=1}^k\binom{k+1}{i}\Delta^{k-i+1}p(x+i)\Delta^iq(x)\\
&=\sum_{i=0}^{k+1}\binom{k+1}{i}\Delta^{k+1-i}p(x+i)\Delta^iq(x),
\end{align*}
completing the inductive step.
\end{proof}
\section{Combinatorial proof of Lemma \ref{binomial-coeffs-of-binomial-coeffs-1}}\label{x-in-R+y-then-R+x-Section}
In this section we prove the following result (Lemma \ref{binomial-coeffs-of-binomial-coeffs-1}):
\begin{lemma}\label{binomial-coeffs-of-binomial-coeffs}
If $\b\in R_+(\a),$ then $R_+(\b)\subset R_+(\a).$
\end{lemma}
This was first proved by Kalinov and Harman in \cite{classify-symmetric-monoidal-functors} by composing the corresponding symmetric monoidal functors. Here, we provide a purely combinatorial proof.
\begin{proof}
Note that $R_+(\a)$ is equal to the image of $R_+(x)$ by the evaluation homomorphism sending $x$ to $\a.$ Hence, it is sufficient to prove a similar fact for $x$ an indeterminate. Specifically, we need to prove that for any $f(x)\in R_+(x),$ the binomial coefficient $\binom{f(x)}{k}\in R_+(x).$ The additive monoid $R_+(f(x))$ generated by the $\binom{f(x)}{k}$ is then contained in $R_+(x),$ as desired. If we expand $\binom{f(x)}{k}$ as a linear combination of binomial coefficients of $x,$ we see that it is sufficient to show that the coefficients of the linear combination are nonnegative integers. Lemma \ref{expansion-into-binomial-coefficients} gives such an expansion, and shows that it is sufficient to prove that the differences $\Delta^n\binom{f(x)}{k}(0)$ are nonnegative integers. By the formula in Lemma \ref{iterated-finite-difference}, it is sufficient to prove that $$
\sum\limits_{i=0}^n(-1)^i\binom{n}{i}\binom{f(n-i)}{k}
$$
is a nonnegative integer. But it is an integer because $f(x)$ is integer-valued for integral inputs and hence all $\binom{f(n-i)}{k}$ are integers, so it is sufficient to prove that the expression above is nonnegative. 

Say $f(x)=a_0+a_1\binom{x}{1}+\cdots+a_t\binom{x}{t}.$ Note that, for $n$ an integer and $s>n,$ all the terms $a_s\binom{n}{s}$ for $s>n$ are removed because $\binom{n}{s}=0$ for any $s>n.$ Therefore, we can write \begin{align*}
    f(n-i)&=a_0+a_1(n-i)+a_2\binom{n-i}{2}+\cdots+a_t\binom{n-i}{t}\\
    &=a_0+a_1(n-i)+a_2\binom{n-i}{2}+\cdots+a_{n-i-1}(n-i)+a_{n-i},
\end{align*} where $a_k=0$ for $k>t.$ Notably, we wish to show that, for any nonnegative integers $a_0,a_1,\dots,a_n,$ $$
\sum\limits_{i=0}^n(-1)^i\binom{n}{i}\bigbinom{\sum\limits_{j=0}^{n-i}a_j\binom{n-i}{j}}{k}\geq0.
$$
This expression can be interpreted as a purely combinatorial formula, so we prove it in this way.

Define $I=\{1,\dots,n\}.$ For any sets $S_1,S_2,\dots,S_n,$ the set $S_1\cup\cdots\cup S_n$ can be represented as the disjoint union of sets $$
S_1\cup\cdots\cup S_n = \coprod_{k=1}^n\coprod_{\{b_1,\dots,b_k\}\subset I}(S_{b_1}\cap S_{b_2}\cap\cdots\cap S_{b_k})\backslash\bigcup_{m\not\in\{b_1,\dots,b_k\}}S_m.
$$
This formula is valid because for any $x\in S_1\cup\cdots\cup S_n,$ $x$ belongs to a unique set in the disjoint union. Specifically, take $S_{b_1},S_{b_2},\dots,S_{b_k}$ to be the sets containing $x$ so that $x$ belongs to the intersection of the sets $S_{b_j}$ but not to the other sets $S_m.$

Similarly, the set $S_i$ is also the disjoint union of all such sets involving $S_i$ in the intersection, specifically $$
S_i=\coprod_{k=0}^{n-1}\coprod_{\{b_1,\dots,b_k\}\subset I\backslash\{i\}}(S_i\cap S_{b_1}\cap S_{b_2}\cap\cdots\cap S_{b_k})\backslash\bigcup_{m\not\in\{i,b_1,\dots,b_k\}}S_m.
$$
Now we would like to define $S_1,\dots,S_n$ by specifying the cardinality of each set in the disjoint union (we don't care about the nature of the elements here, only about the cardinality of the sets). Let each set $(S_{b_1}\cap S_{b_2}\cap\cdots\cap S_{b_k})\backslash\{S_m\mid m\not\in\{b_1,\dots,b_k\}\}$ have exactly $a_{n-k}$ elements, where $a_i$ are the coefficients of the binomial expansion of $f(x).$ Again, it is not important to us what these elements are, simply that they are distinct. See Figure \ref{fig:venn-diagram} for the case when $n=3.$
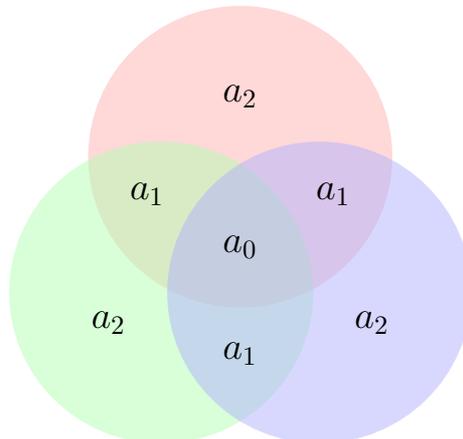
\begin{figure}[H]
    \centering
\begin{tikzpicture}
\begin{scope}[opacity=0.5]
\fill[red!30!white]   ( 90:1.2) circle (2);
    \fill[green!30!white] (210:1.2) circle (2);
    \fill[blue!30!white]  (330:1.2) circle (2);
\end{scope}
\node at ( 90:2)  [font=\large]  {$a_2$};
  \node at (210:2)  [font=\large]  {$a_2$};
  \node at (330:2)  [font=\large]  {$a_2$};
  \node at (30:1.414)  [font=\large]  {$a_1$};
  \node at (150:1.414)  [font=\large]  {$a_1$};
  \node at (270:1.414)  [font=\large]  {$a_1$};
  \node [font=\large] {$a_0$};
\end{tikzpicture}
\caption{Sets $S_1,S_2,S_3$ and the number of elements in each intersection}
\label{fig:venn-diagram}
\end{figure}

\noindent For any $p\leq n,$ we see that the set $S_1\cap S_2\cap\cdots\cap S_p$ has exactly $$
a_{n-p}+(n-p)a_{n-p-1}+\binom{n-p}{2}a_{n-p-2}+\cdots+\binom{n-p}{2}a_2+(n-p)a_1+a_0
$$
elements. This is because a nonempty set $(S_{b_1}\cap S_{b_2}\cap\cdots\cap S_{b_k})\backslash\bigcup\limits_{m\not\in\{b_1,\dots,b_k\}}S_m$ is contained in $S_1\cap S_2\cap\cdots\cap S_p$ if and only if $\{1,2,\dots,p\}\subset\{b_1,b_2,\dots,b_k\}.$ Therefore, the number of such sets containing $a_{n-p-k}$ elements is exactly $\binom{n-p}{n-k-p}=\binom{n-p}{k}.$ In addition, the set $S_1\cup S_2\cup\cdots\cup S_n$ has exactly $$
na_{n-1}+\binom{n}{2}a_{n-2}+\cdots+\binom{n}{2}a_2+na_1+a_0
$$
elements, because for any $k,$ there are $\binom{n}{k}$ sets $(S_{b_1}\cap S_{b_2}\cap\cdots\cap S_{b_k})\backslash\{S_m\mid m\not\in\{b_1,\dots,b_k\}\}$ corresponding to the different choices of $\{b_1,\dots,b_k\}\subset\{1,\dots,n\}.$

Consider the number of ways to choose $k$ elements out of $S_1\cup S_2\cup\cdots\cup S_n$ such that they all lie in a single one of the $S_i.$ On the one hand, this is bounded above by the total number of ways to choose $k$ elements out of $S_1\cup S_2\cup\cdots\cup S_n,$ equal to $$
\bigbinom{na_{n-1}+\binom{n}{2}a_{n-2}+\cdots+\binom{n}{2}a_2+na_1+a_0}{k}.
$$
Note that this, in turn, is bounded above by $$
\binom{f(n)}{k}=\bigbinom{a_n+na_{n-1}+\cdots+\binom{n}{2}a_2+na_1+a_0}{k}.
$$
On the other hand, by the principle of inclusion-exclusion, it is equal to \begin{align*}
&\sum_{b\in I}\bigbinom{\abs{S_b}}{k}-\sum_{b_1,b_2\in I}\bigbinom{\abs{S_{b_1}\cap S_{b_2}}}{k}+\cdots+(-1)^{n+1}\sum_{b_1,\dots,b_n\in I}\bigbinom{\abs{S_{b_1}\cap\cdots\cap S_{b_n}}}{k}\\
=&n\bigbinom{a_{n-1}+(n-1)a_{n-2}+\cdots+a_0}{k}-\cdots+n(-1)^{n-2}\bigbinom{a_1+a_0}{k}+(-1)^{n-1}\bigbinom{a_0}{k}\\
=&n\binom{f(n-1)}{k}-\binom{n}{2}\binom{f(n-2)}{k}+\cdots+n(-1)^{n-2}\binom{f(1)}{k}+(-1)^{n-1}\binom{f(0)}{k}.
\end{align*}
We subtract this from $\binom{f(n)}{k}$ to obtain the inequality $$
\binom{f(n)}{k}-n\binom{f(n-1)}{k}+\binom{n}{2}\binom{f(n-2)}{k}-\cdots+n(-1)^{n-1}\binom{f(1)}{k}+(-1)^n\binom{f(0)}{k}\geq0,
$$
as desired.
\end{proof}
\section{Ring structure}\label{Ring-Section}
In this section we prove (the nontrivial direction of) Theorem \ref{intro-its-a-ring}. The logic of the proof is to start with the minimal polynomial of the algebraic number, and then to prove that we can obtain an expression for $-1$ in terms of positive integer linear combinations of $\binom{\a}{k}$ using multiples of the polynomial. We start with a few basic results and definitions.
\begin{lemma}\label{integral-in-normal-basis-is-integral-in-binom}
If $p(x)$ is a polynomial with integral coefficients with $p(x)=\sum\limits_{i=0}^na_i\binom{x}{i},$ then the $a_i$ are integers.
\end{lemma}
\begin{proof}
This follows immediately from the fact that $R_+(x)\subset\Q[x]$ is closed under multiplication, hence $x^k\in R_+(x)$ is a linear combination of the $\binom{x}{i}$ with nonnegative integral coefficients.
\end{proof}
Now, we provide a useful definition. Recall from Definition \ref{difference-operator-definition} that $\Delta$ is the difference operator acting on the set of polynomials.
\begin{definition}\label{binomial-coefficient-derivative}
For a polynomial $p(x),$ define $p_{nk}=\Delta^n\left(x^kp(x)\right)(0).$ In other words, these are the coefficients for the binomial expansion of $x^kp(x).$ In particular, define $g_{nk}(i)=\binom{x}{i}_{nk}=\Delta^n\left(x^k\binom{x}{i}\right)(0).$
\end{definition}
Note that if we have a polynomial $p(x)=a_0+a_1x+\cdots+a_t\binom{x}{t},$ then \begin{align*}
    &p(x)x^k=p_{0k}+p_{1k}x+\cdots+p_{(k+t)k}\binom{x}{k+t}\\
    =(&a_0g_{0k}(0)+a_1g_{0k}(1)+\cdots+a_tg_{0k}(t))+(a_0g_{1k}(0)+a_1g_{1k}(1)+\cdots+a_tg_{1k}(t))x+\cdots\\ +(&a_0g_{(k+t)k}(0)+a_1g_{(k+t)k}(1)+\cdots+a_tg_{(k+t)k}(t))\binom{x}{k+t},
\end{align*} 
which gives us $p_{ik}=a_0g_{ik}(0)+a_1g_{ik}(1)+\cdots+a_tg_{ik}(t),$ demonstrating the usefulness of the definition.
\begin{lemma}\label{bound-on-next-i}
Suppose $g_{nk}(i)$ are as in Definition \ref{binomial-coefficient-derivative}. Then for any nonnegative integers $n,k$ and positive integer $i,$ $g_{nk}(i)\geq\frac{n-i+1}{i}(g_{nk}(i-1)).$
\end{lemma}
\begin{proof}
Define, for simplicity, $m!=\infty$ (so $\frac{1}{m!}=0)$ for $m<0.$ Then, from Lemma \ref{product-rule-for-difference-quotient}, we immediately obtain \begin{align*}
g_{nk}(i)&=\left(\sum_{j=0}^n\binom{n}{j}\Delta^{n-j}\binom{x+j}{i}\Delta^j(x^k)\right)(0)=\left(\sum_{j=0}^n\binom{n}{j}\binom{x+j}{i+j-n}\Delta^j(x^k)\right)(0)\\
&=\sum_{j=0}^n\binom{n}{j}\binom{j}{i+j-n}g_{jk}(0)=\sum_{j=0}^n\frac{n!}{(n-j)!(i+j-n)!(n-i)!}g_{jk}(0)\\
&=\frac{n!}{(n-i)!}\sum_{j=0}^n\frac{g_{jk}(0)}{(n-j)!(i+j-n)!}\geq\frac{n!}{(n-i)!}\sum_{j=0}^n\frac{g_{jk}(0)}{i(n-j)!(i+j-n-1)!}\\
&=\left(\frac{n-i+1}{i}\right)\frac{n!}{(n-i+1)!}\sum_{j=0}^n\frac{g_{jk}(0)}{(n-j)!(i+j-n-1)!}=\left(\frac{n-i+1}{i}\right)g_{nk}(i-1),
\end{align*}
as desired.
\end{proof}
This gives us an important fact about $p_{nk}.$
\begin{lemma}\label{asymptotic-N}
If the leading coefficient $a_t$ of $p(x)$ is positive, then there exists $N$ such that for all $n\geq N$ and $k\geq0,$ the value $p_{nk}\geq0.$ In particular, for $n>N,$ we have $p_{nk}\geq np_{n(k-1)}.$
\end{lemma}
\begin{proof}
We use the formula $p_{nk}=a_0g_{nk}(0)+a_1g_{nk}(1)+\dots+a_tg_{nk}(t).$ Thus, it is sufficient to find $N$ such that for all $n\geq N,$ $$
a_tg_{nk}(t)\geq\sum\limits_{i=0}^{t-1}\abs{a_i}g_{nk}(i),
$$
as then $p_{nk}$ must be at least $0.$ But the $a_i$ are constants and $t,$ the degree of $p(x),$ is as well. We use Lemma \ref{bound-on-next-i} and note that $i\leq t$ and $t$ is constant, to get, for $n\geq2t>i+t-1,$ that $$
g_{nk}(i)\leq\frac{t(t-1)\dots(i+1)}{(n-t+1)(n-t+2)\dots(n-i)}g_{nk}(t)\leq\frac{tg_{nk}(t)}{n-t+1}.
$$
We note that $a_t$ is a positive integer to get that $$
\sum\limits_{i=0}^{t-1}\abs{a_i}g_{nk}(i)\leq\frac{t^2g_{nk}(t)\max_i\abs{a_i}}{n-t+1}\leq a_tg_{nk}(t)
$$
for $n$ large enough, as desired. In particular, since $a_t$ is a positive integer, we can take $N=t^2\max\limits_i\abs{a_i}+2t.$ 

Now, for the second part of the Lemma, note that, for $n>N,$ $n-1\geq N$ so $p_{(n-1)(k-1)}\geq0.$ Now, we note the remarkable identity $x\binom{x}{n-1}=(n-1)\binom{x}{n-1}+n\binom{x}{n}.$ We conclude from this that $$
p_{nk}=n\left(p_{n(k-1)}+p_{(n-1)(k-1)}\right)\geq np_{n(k-1)},
$$
as desired.
\end{proof}
If we define $c_n=p(n),$ we can obtain an alternative formula for $p_{nk}.$ 
\begin{lemma}\label{alt-formula-for-coeffs}
$p_{nk}=c_nn^k-nc_{n-1}(n-1)^k+\binom{n}{2}c_{n-2}(n-2)^k-\cdots+(-1)^nc_00^k,$ where $0^0=1.$
\end{lemma}
\begin{proof}
We substitute $p_{nk}=\Delta^n(x^kp(x))(0)$ into Lemma \ref{iterated-finite-difference} to get $$
p_{nk}=\left(\sum\limits_{i=0}^n(-1)^i\binom{n}{i}(x+n-i)^kp(x+n-i)\right)(0)=\sum\limits_{i=0}^n(-1)^i\binom{n}{i}(n-i)^kc_{n-i},
$$
which is exactly the formula in the statement of the lemma.
\end{proof}
From this formula the following asymptotic description follows:
\begin{lemma}\label{asymptotic-behavior}
We have $p_{nk}\sim c_nn^k$ in the sense that, as $k\to\infty,$ the quotient $\frac{p_{nk}}{c_nn^k}\to1.$
\end{lemma}
\begin{proof}
This is trivial by Lemma \ref{alt-formula-for-coeffs} and by noting the fact that $$
n^k=\left(\frac{n}{n-i}\right)^k(n-i)^k\gg(n-i)^k
$$
for a fixed $n,$ $i\leq n,$ and $k\to\infty.$
\end{proof}
We finally have one more important lemma to prove.
\begin{lemma}\label{linear-independence-of-coeffs}
If $p(n)\neq0$ for $0\leq n\leq N,$ then the $N+1$ vectors $\left<p_{0k},p_{1k},\dots,p_{Nk}\right>,$ where $k=0,\dots,N,$ are linearly independent.
\end{lemma}
\begin{proof}
Say they aren't. Then some nontrivial linear combination of the vectors is equal to zero. But note then that some linear combination of the polynomials $x^kp(x),0\leq k\leq N,$ is zero evaluated at each of $0,1,\dots,N,$ since each such polynomial is equal to $$
p_{0k}+p_{1k}x+p_{2k}\binom{x}{2}+\dots+p_{Nk}\binom{x}{N}+\cdots.
$$
Thus, there exists some polynomial $q(x)$ of degree at most $N$ such that $p(n)q(n)=0$ for $0\leq n\leq N.$ But since $p(n)\neq0,$ this means that $q(n)=0$ for each such $n.$ This implies that $q$ is identically $0,$ contradiction.
\end{proof}
In fact, it is possible to prove by Gaussian elimination that the determinant of the matrix formed by these vectors is equal to $\prod\limits_{n=0}^Nc_nn!,$ which is obviously nonzero.
\begin{corollary}\label{int-combination}
For any vector $\left<v_0,v_1,\dots,v_N\right>\in\Q^n$ (i.e. $v_1,\dots,v_N \in \Q),$ there exists a positive integer $m$ such that $\left<mv_0,mv_1,\dots,mv_N\right>$ can be expressed as an integer linear combination of the vectors from Lemma \ref{linear-independence-of-coeffs}. 
\end{corollary}
\begin{proof}
$\left<v_0,v_1,\dots,v_N\right>$ can be expressed as a rational linear combination of those $N+1$ vectors, so we can multiply through by the denominators of the coefficients to obtain the desired result.
\end{proof}
We are now ready to prove the main theorem of this section.
\begin{theorem}\label{multiply-to-get-pos-coeffs}
For any polynomial $p(x)$ with rational coefficients and with $p(n)\neq0$ for all nonnegative integers $n,$ there exists a polynomial $s(x)$ such that $\Delta^n(p(x)s(x))(0)\geq0$ for all $n$ and $p(0)s(0)>0.$
\end{theorem}
\begin{proof}
Say $s(x)=b_0+b_1x+\dots+b_jx^j.$ Then note that $$
\Delta^n(p(x)s(x))(0)=b_0p_{n0}+b_1p_{n1}+\dots+b_jp_{nj}.
$$
Hence, it suffices to find a linear combination of the $p_{nk}$ that is positive for $n=0$ and nonnegative for any $n>0.$

Assume $p(x)=a_t\binom{x}{t}+a_{t-1}\binom{x}{t-1}+\dots+a_1x+a_0,$ and say $a_t>0$ (otherwise multiply by $-1).$ Fix $N$ to satisfy Lemma \ref{asymptotic-N}. By Lemma \ref{asymptotic-behavior}, for any choice of real numbers $c_{n_-},c_{n_+}$ such that $c_{n_-}<\abs{c_n}<c_{n_+},$ for large enough $K_n(c_{n_-},c_{n_+}),$ for any $k>K_n(c_{n_-},c_{n_+})$ we have $c_{n_-}n^k<\abs{p_{nk}}<c_{n_+}n^k.$ Now, for every $n\leq N,$ choose $c_{n_-},c_{n_+}$ with $c_{n_-}<\abs{c_n}<c_{n_+}$ such that $\frac{c_{n_+}}{c_{n_-}}<\frac{N+1}{N}$; this is possible because $\frac{N+1}{N}>1.$ For example, take $c_{n_+}=\frac{2N+1}{2N}\abs{c_n}$ and $c_{n_-}=\frac{4N+2}{4N+3}\abs{c_n}.$ Denote $$
\KK=\max(\max\limits_{n\leq N}K_n(c_{n_-},c_{n_+}),N+2)+1.
$$
Then we notice that for any $k>\KK-1$ and $n\leq N,$ we have $c_{n_-}n^k<\abs{p_{nk}}<c_{n_+}n^k.$ We divide by the previous inequality to get, for any $k>\KK$ and $n\leq N,$ that $\abs{\frac{p_{nk}}{p_{n(k-1)}}}<\frac{c_{n_+}}{c_{n_-}}n,$ and we rearrange to get $$
\abs{p_{nk}}<\abs{p_{n(k-1)}}\frac{N+1}{N}n\leq\abs{p_{n(k-1)}}(N+1).
$$
Thus, for any $k>\KK$ and $n\leq N,$ we have \begin{equation}\label{eqn:upper-bound}
    \abs{p_{nk}}<\abs{p_{n\KK}}(N+1)^{k-\KK}.
\end{equation}
By Lemma \ref{linear-independence-of-coeffs}, there exists a linear combination of the $\left<p_{nk}\right>$ for $0\leq k\leq N+2$ to get a vector $\left<r_n\right>$ with $r_0=m$ and $r_n=m\abs{p_{n\KK}}$ for $0<n\leq N+2.$ If $q_l$ are coefficients of this linear combination, i.e. we have  $\sum_lq_l\overline{p_l}=\overline{r}$ for $\overline{p_l}= \left<p_{0l},\dots,p_{(N+2)l}\right>$ and $\overline{r}=\left<r_0,\dots,r_{N+2}\right>$, then $q(x)=\sum_lq_lx^l.$ This therefore gives us a polynomial $q(x)$ of degree at most $N+2$ such that the first $N+3$ coefficients in the binomial expansion of $r(x)=p(x)q(x)$ are positive, with the constant term equal to $m$ and the other $N+2$ coefficients equal to $m\abs{p_{n\KK}}.$ This almost gives us what we want! The polynomial $p(x)$ has degree $t,$ so the polynomials $p(x)x^k$ for $k\leq N+2$ have degree at most $N+2+t.$ Thus, only the $t$ terms of highest degree may possibly have negative coefficients. We now resolve this issue.

Say $R=\max\limits_{N+2<n\leq N+t+2}\abs{r_n}.$ Note that $R$ is a constant. Now, for $N+2<n\leq N+t+2,$ $p_{nk}$ is nonnegative and $$
p_{nk}\geq np_{n(k-1)}>(N+2)p_{n(k-1)}>(N+2)^{k-\KK}p_{n\KK}=\left(\frac{N+2}{N+1}\right)^{k-\KK}(N+1)^{k-\KK}p_{n\KK}
$$
for $k\geq\KK.$ Note that $p_{n\KK}$ is nonzero because $\KK>N+2$ and $p_{n(n-t)}$ is nonzero, so $p_{nk}>0$ for any $k\geq n-t$ (which is less than or equal to $N+2),$ so this inequality makes sense. Now, it is important to note that $N$ is constant, therefore $\frac{N+2}{N+1}$ is a fixed constant greater than $1.$ Thus, as $k\to\infty,$ $\left(\frac{N+2}{N+1}\right)^{k-\KK}$ grows arbitrarily large, in particular there exists a $K$ such that $\left(\frac{N+2}{N+1}\right)^{K-\KK}>R.$ We thus have \begin{equation}\label{eqn:lower-bound}
    \abs{p_{nK}}\geq R(N+1)^{K-\KK}p_{n\KK}.
\end{equation}

Now, consider the polynomial $s(x)=q(x)(N+1)^{K-\KK}+x^K.$ We claim that $s(x)$ satisfies the conditions of the theorem. First, $$
p(0)s(0)=p(0)q(0)(N+1)^{K-\KK}=r(0)(N+1)^{K-\KK}=m(N+1)^{K-\KK}>0,
$$
as desired. Now, take an arbitrary $n.$ Then we wish to show that \begin{align*}
\Delta^n(p(x)s(x))(0)&=\Delta^n(p(x)q(x)(N+1)^{K-\KK}+p(x)x^K)(0)\\
&=(N+1)^{K-\KK}\Delta^nr(x)(0)+\Delta^n(p(x)x^K)(0)=(N+1)^{K-\KK}r_n+p_{nK}\geq0.
\end{align*}
Suppose first that $0<n\leq N.$ Then, because $r_n=m\abs{p_{n\KK}}>0,$ it suffices to show that $$
\abs{p_{nK}}\leq(N+1)^{K-\KK}r_n=m(N+1)^{K-\KK}\abs{p_{n\KK}}.
$$
But $m\geq1,$ and we have from Equation (\ref{eqn:upper-bound}) that $\abs{p_{nK}}<\abs{p_{n\KK}}(N+1)^{K-\KK},$ as desired. Now, say $n=N+1$ or $n=N+2.$ Then $r_n=m\abs{p_{n\KK}}>0$ and $p_{nK}\geq0,$ so $$
(N+1)^{K-\KK}r_n+p_{nK}>0,
$$
as desired. Suppose $N+2<n\leq N+t+2.$ Then, by Equation (\ref{eqn:lower-bound}) we have $$
p_{nK}\geq R(N+1)^{K-\KK}p_{n\KK}>0.
$$
It thus suffices to show that $(N+1)^{K-\KK}\abs{r_n}\leq p_{nK},$ but this follows by noting that $R=\max\limits_{N+2<n\leq N+t+2}\abs{r_n}$ and $p_{n\KK}\geq1$ since it is a positive integer, as desired. Finally, say $n>N+t+2.$ Then $r_n=0$ $(r(x)$ has degree at most $N+t+2),$ so $$
(N+1)^{K-\KK}r_n+p_{nK}=p_{nK}\geq0,
$$
as desired.

Thus, in any case, $\Delta^n(p(x)s(x))(0)\geq0$ for all $n$ and $p(0)s(0)>0,$ as desired.
\comment{
this was an early ugly draft of the proof
\todo{REVISE THIS PART A LOT}
By Lemma \ref{linear-independence-of-coeffs}, there exists a linear combination of the $p_{nk}$ for $0\leq k\leq N+1$ to get a sequence $r_n$ with $r_n=mp_{n\k}$ for $0\leq n\leq N.$ Now, note that $p(x)x^k$ has degree $k+t\leq N+t.$ Hence, the polynomial $r(x)$ formed by the coefficients $r_n$ has degree at most $N+t.$ Say $\a=\max\limits_{N<n\leq N+t}\abs{r_n}.$ Now, note that, for any $N+1<n\leq N+t,$ $p_{nk}\geq np_{n(k-1)}\geq(N+1)^{k-n-t}p_{n(n+t)}=\frac{(n+t)!}{t!}a_t(N+1)^{k-n-t}=\frac{(n+t)!}{t!}a_t(N+0.5)^{k-n-t}\left(\frac{N+1}{N+0.5}\right)^{k-n-t}\geq\frac{(n+t)!}{t!}a_t(N+0.5)^{k-n-t}\left(\frac{N+1}{N+0.5}\right)^{k-N-1-t}.$ Obviously, $N,t$ are constants independent of $k,$ so we can fix a sufficiently large $K$ so that $\a\leq\frac{(n+t)!}{t!}a_t\left(\frac{N+1}{N+0.5}\right)^{K-N-1-t}.$ Thus, we have $p_{nK}\geq\a(N+0.5)^{K-n-t}\geq\abs{r_n}(N+0.5)^{K-n-t}\geq\abs{r_n}(N+0.5)^{K-N-1-2t}$ for $N+1<n\leq N+t.$ On the other hand, for $n\leq N,$ we have $p_{nK}<p_{n\k}(N+0.5)^{K-\k}=\frac{r_n}{m}(N+0.5)^{K-\k}\leq r_n(N+0.5)^{K-\k}<r_n(N+0.5)^{K-N-1-2t}.$ Thus, we have $p_{nK}+r_n(N-0.5)^{K-N-2t}$ is greater than $0$ for $n\leq N$ (in particular, $n=0),$ and is nonnegative for $N+1<n\leq N+t.$ But for $n=N+1,$ both terms are positive so this is obviously true, and for any $n>N+t,$ $r_n=0$ and $p_{nK}\geq0$ obviously. Thus, $p_{nK}+r_n(N-0.5)^{K-N-2t}$ is nonnegative for all $n,$ and positive for $n=0,$ and this completes the proof.
}
\end{proof}
\begin{corollary}\label{its-a-ring}
$-1\in R_+(\a)$ for any non-nonnegative integral algebraic number $\a.$ Hence, $R_+(\a)$ is a ring.
\end{corollary}
\begin{proof}
Let $p(x)$ be the minimal polynomial of $s.$ Note that $p(n)\neq0$ for any nonnegative integer $n$ because otherwise $p(x)$ is divisible by $(x-n)$ and hence is not irreducible. Thus, by Theorem \ref{multiply-to-get-pos-coeffs}, there exists a polynomial $q(x)$ such that $\Delta^n(p(x)q(x))\geq0$ for all $n$ and $p(0)q(0)>0.$ But then note that the binomial expansion of this product $$
r(x)=p(x)q(x)=a_t\binom{x}{t}+a_{t-1}\binom{x}{t-1}+\cdots+a_1x+a_0
$$
has $a_0=p(0)q(0)>0,a_1,\dots,a_t\geq0.$ But note that $r(\a)=p(\a)q(\a)=0.$ Thus, we have $$
-a_0=a_t\binom{\a}{t}+a_{t-1}\binom{\a}{t-1}+\cdots+a_1\a\in R_+(\a)
$$
because the $a_i$ are nonnegative integers. But then we have $(a_0-1)-a_0=-1\in R_+(\a),$ as desired.
\end{proof}
\vspace{-.3in}
\section{Description of \texorpdfstring{$R_+(\alpha)$}{} in terms of valuations}\label{General-Section}
In this section we determine an explicit formula for $R_+(\a),$ for any algebraic number $\a,$ in terms of the valuations with respect to the prime ideals of $\OO_{\Q(\a)}$.


\subsection{Integral elements of \texorpdfstring{$R_+(\a)$}{}}\label{section-integral-elements}
Since each $\binom{\alpha}{k}$ lies in $\Q(\a)$ we know that $R_+(\a)\subset\Q(\a)$. First, we would like to focus on specifying which {\it algebraic integers} in $\Q(\a)$ also lie in $R_+(\a).$ Specifically, an algebraic integer is any number that is a root of a monic polynomial with integral coefficients. The integers of a field form a ring, called the {\it ring of integers} $\OO_{\Q(\a)}.$

A very important concept for general $\a$ is the {\it norm}, defined for an element $\b\in\Q(\a)$ as $$
N_\a(\b)=\det(\m_\b),
$$
where $\m_\b$ is the $\Q$-linear transformation of $\Q(\a)$ defined by multiplication by $\b,$ i.e. $x\mapsto\b x.$ Because the determinant does not change under a change of basis and is multiplicative, the norm is well-defined and multiplicative. In addition, the norm of an element $\b$ is a perfect power of the product of the conjugates of $\b,$ which is also the constant term of the minimal polynomial of $\b.$ Specifically, the power is $\left[\Q[\a]:\Q[\b]\right].$ The following is an important property of the norm of an element $\b\in R_+(\a).$
\begin{lemma}\label{norm-over-element-in}
If $\b$ is a nonzero algebraic integer and $\b\in R_+(\a),$ then $\frac{N_\a(\b)}{\b}\in R_+(\a).$
\end{lemma}
\begin{proof}
Say the minimal polynomial of $\b$ is $m_\b(x)=x^n+a_{n-1}x^{n-1}+\dots+a_1x+a_0,$ where $a_i\in\Z$ for $i=0,1,\dots,n-1.$ Note that $\b$ is a root, hence $\b^n+a_{n-1}\b^{n-1}+\dots+a_1\b+a_0=0.$ Thus, $a_0=b_0\b=-(\b^n+a_{n-1}\b^{n-1}+\dots+a_1\b).$ We divide by $\b$ to get $$
b_0=-(\b^{n-1}+a_{n-1}\b^{n-2}+\dots+a_1)\in\Z[\b].
$$
But we have, for some positive integer power $k,$ and up to multiplication by $-1$ (which is possible, because we are in a ring), that $$
\frac{N_\a(\b)}{\b}=\frac{a_0^k}{\b}=b_0(a_0^{k-1})\in\Z[\b]\subset R_+(\alpha),
$$
as desired.
\end{proof}
We use this lemma to show that all elements in the ring of integers belong to $R_+(\a)$, but first we need to prove another lemma.
\begin{lemma}\label{mult-of-alg-int}
If $\b\in\OO_{\Q(\a)},$ then there exists a positive integer $\d$ such that $\d\b\in R_+(\a).$
\end{lemma}
\begin{proof}
By Corollary \ref{its-a-ring} of Theorem \ref{multiply-to-get-pos-coeffs}, $R_+(\a)$ is a ring, hence $\Z[\a]\subset R_+(\a).$ Now, suppose $\b=a_0+a_1\a+\dots+a_k\a^k.$ Then the $a_i$ are rational, hence we can multiply through by their common denominator $m$ to get, with $b_i\in\Z,$ $m\b=b_0+b_1\a+\dots+b_k\a^k\in\Z[\a],$ as desired.
\end{proof}
Now we are ready to prove the main result of this subsection. 
\begin{prop}\label{ring-of-integers}
For all algebraic numbers $\a\not\in\Z_{\geq0},$ the ring of integers $\OO_{\Q(\a)}\subset R_+(\a).$
\end{prop}
\begin{proof}
It is sufficient to show that, for any algebraic integer $\b\in\OO_{\Q(\a)},$ we have $\b\in R_+(\a).$ By Lemma \ref{mult-of-alg-int}, there exists a positive integer $\d$ such that $\d\b\in R_+(\a).$ Consider the smallest such $\d.$ Our goal is to show that $\d=1.$

We show that, if a prime $p$ divides $\d,$ then $(\d/p)\b\in R_+(\d\b)\subset R_+(\a)$ (the last step is true by Lemma \ref{binomial-coeffs-of-binomial-coeffs}), which immediately implies the conclusion because then the smallest $\d$ such that $\d\b\in R_+(\a)$ must not be divisible by any primes, and hence must be equal to $1.$

Say that $\b$ is of degree $n,$ with minimal polynomial $x^n+a_{n-1}x^{n-1}+\cdots+a_1x+a_0,$ with $a_i\in\Z.$ For any prime $p$ dividing $\d,$ consider the element $(p-1)!\binom{\d\b}{p}\in R_+(\d\b).$ We expand, noting that $\d$ is divisible by $p,$ to get that $$
(\d/p)\b(\d\b-1)(\d\b-2)\cdots(\d\b-p+1)\in R_+(\d\b).
$$
Now, we apply Lemma \ref{norm-over-element-in} on each of $\d\b-i$ (these are algebraic integers) to get that $$
(\d/p)\b N_\b(\d\b-1)N_\b(\d\b-2)\cdots N_\b(\d\b-p+1)\in R_+(\d\b).
$$
But we have, where $\b_j$ are the conjugates of $\b,$ \begin{align*}
N_\b(\d\b-i)&=\prod_{j=1}^n(\d\b_j-i)=(-\d)^n\prod_{j=1}^n\left(\frac{i}{\d}-\b_j\right)=(-\d)^nm_\b(i/\d)\\
&=(-\d)^n\left(\left(\frac{i}{\d}\right)^n+a_{n-1}\left(\frac{i}{\d}\right)^{n-1}+\cdots+a_1\left(\frac{i}{\d}\right)+a_0\right)\\
&=(-1)^n\left(i^n+\d a_{n-1}i^{n-1}+\cdots+\d^{n-1}a_1i+\d^na_0\right)\equiv(-i)^n\not\equiv0\pmod p,
\end{align*}
where $m_\b$ is the minimal polynomial of $\b.$ Thus, we have $(\d k/p)\b\in R_+(\d\b)$ where $k$ is an integer not divisible by $p.$ But then $\gcd(\d,\d k/p)=\d/p,$ in particular Bezout's identity implies that there exists a linear combination equal to $\d/p.$ Now, $\d\b,(\d k/p)\b\in R_+(\d\b),$ so their greatest common divisor $(\d/p)\b\in R_+(\d\b)$ as well, as desired.
\end{proof}
By Proposition \ref{ring-of-integers}, $\OO_{\Q(\a)}\subset R_+(\a).$ We then know that for any $\b\in\OO_{\Q(\a)},$ $R_+(\b)\subset R_+(\a).$ If $\b$ has the same degree as $\a,$ then $\OO_{\Q(\b)}=\OO_{\Q(\a)}.$ But then $\a\in\OO_{\Q(\b)}\subset R_+(\b),$ so $R_+(\a)\subset R_+(\b).$ Hence, $R_+(\b)=R_+(\a),$ so \textit{$R_+(\b)$ does not depend on the choice of $\b\in\OO_{\Q(\a)},$} as long as it spans the same ring of integers. This motivates the following more general definition of $R_+$ for a ring:
\begin{definition}
For a ring $S,$ define $R_+(S)$ as the set $$
\bigplus_{s\in S}R_+(s)=\{t_1+t_2+\dots+t_n\mid t_i\in R_+(s_i),s_i\in S\}.
$$
In other words, $R_+(S)$ is the set of nonnegative linear combinations with integral coefficients of binomial coefficients of elements of $S.$
\end{definition}
Note that $S\subset R_+(S).$ It is not obvious whether $R_+(S)$ is a ring, since it is defined as a set of linear combinations of elements of different rings, so it may not be closed under multiplication. However:
\begin{prop}\label{R+-of-ring-is-a-ring}
For any ring $S,$ the set $R_+(S)$ is a ring.
\end{prop}
We prove this in Lemma \ref{R+-of-ring-is-ring}.

For any $\a\in S$ such that $S\subset R_+(\a),$ we see that $R_+(S)=R_+(\a).$ In particular, for algebraic integers $\a,$ $R_+(\OO_{\Q(\a)})=R_+(\a).$ In general, we have $R_+(\OO_{\Q(\a)}[\a])=R_+(\a).$ 
\subsection{General results on overrings of \texorpdfstring{$\OO_K$}{} }\label{Section-Ideal-Factorization}
In previous sections we proved that $R_+(\a)$ is a ring sandwiched between $\mathcal O_{\mathbb Q(\a)} \subset R_+(\a)\subset \mathbb Q(\a)$. Such objects are called \textbf{overrings} of $\mathcal O_{\mathbb Q(\a)}$.
\begin{definition}
Say $K$ is any field and $S$ is a subring of $K$ containing the ring of integers $\OO_K.$ We call $S$ an overring of $\OO_K.$
\end{definition}
In this section we prove some general results about overrings of $\OO_K$. These results are not new, but we weren't able to find a reference where they would be stated in exactly this way, so we provide our own proofs.

First, recall that in the rational case, an advantage we had was that integers factorize uniquely into products of primes. Unfortunately, in general integral domains, this is not true. For example, in $\Z[\sqrt{-5}],$ the element $6$ factors, on one hand, as $(1+\sqrt{-5})(1-\sqrt{-5}),$ and on the other hand, as $2\cdot3.$ Luckily, there is a form of factorization that works uniquely in rings of integers of number fields, which are Dedekind rings, and that is ideal factorization into prime ideals. Specifically, every ideal (possibly fractional) factorizes uniquely as a product of prime ideals raised to integer powers.\footnote{For a principal ideal $I=(i),$ we have $I^{-1}=(i^{-1}).$ In general for any (possibly non-principal) ideal, we can define $I^{-1}$ as the set of all $k\in K$ such that $Ik\subset\OO_K.$ For example, in $\Q(\sqrt{-5}),$ we have $(1+\sqrt{-5},2)^{-1}=(\frac{1+\sqrt{-5}}{2},1).$} Therefore, in the field $K,$ we can define for a prime ideal $\pp,$ as we did for the rational case, the $\pp$-adic valuation. 
\begin{definition}
For a fixed prime ideal $\pp\subset \OO_K$ and a (possibly fractional) ideal $I\subset K,$ define the $\pp$-adic valuation $v_\pp(I)$ as the exponent of $\pp$ (possibly negative) in the factorization of $I,$ and $v_\pp((0))=\infty.$ For any element $x\in K,$ let $v_\pp(x)=v_\pp((x)),$ the $\pp$-adic valuation of the principal ideal generated by $x.$
\end{definition}
This valuation has many of the same properties as the $p$-adic valuation in the rational numbers. For any two elements $x,y\in K,$ we have $v_\pp(xy)=v_\pp(x)+v_\pp(y)$ and $v_\pp(x+y)\geq\min(v_\pp(x),v_\pp(y)).$ These follow from the facts that the ideal $(xy)=(x)(y)$ and that $(x+y)\subset(x)+(y),$ respectively. Also, $v_\pp(x^{-1})=-v_\pp(x)$ because the ideal $(x^{-1})$ is the inverse of $(x)$ in the group of ideals under multiplication.

In particular, we can now formulate simple conditions for determining, for a fixed algebraic number $\a,$ whether a number $\b\in\Q(\a)$ lies in $R_+(\a)=R_+(\OO_{\Q(\a)}[\a]).$ According to the above, for any prime ideal $\pp,$ if $v_\pp\left(\binom{\alpha}{k}\right)\geq0$ for all $k,$ then $v_\pp(\beta)\geq0$ for any $\beta\in R_+(\alpha)$ (as $\beta$ is an integral linear combination of the $\binom{\alpha}{k}).$ The lemma below is a step towards the converse result.
\begin{lemma}\label{prime-inverse-in}
Consider any number field $K$ and an overring $S$ of $\OO_K.$ For any prime ideal $\pp\subset\OO_K,$ if $v_\pp(x)<0$ for some $x\in S,$ then $\pp^{-1}\subset S.$ 
\end{lemma}
\begin{proof}
Consider the (fractional) ideal $(x)=x\OO_K\subset S$ and its prime factorization, which includes $\pp^{-1}$ by definition of the $\pp$-adic valuation, i.e. $$
(x) = \pp^{-1}\pp_1^{-\a_1}\pp_2^{-\a_2}\cdots\pp_k^{-\a_k}\qq_1^{\b_1}\qq_2^{\b_2}\cdots\qq_\ell^{\b_\ell},
$$
where $\a_i,\b_j\geq0,$ and $\pp\neq\qq_j$ for any $j.$ Since $\OO_K\subset S,$ any prime ideal $\pp_i\subset S$ as well. In particular, since $S$ is a ring, we can multiply by these ideals to get that $$
\pp^{-1}\qq_1^{\b_1}\qq_2^{\b_2}\cdots\qq_\ell^{\b_\ell}\subset S.
$$
Let $Q=\qq_1^{\b_1}\qq_2^{\b_2}\cdots\qq_\ell^{\b_\ell},$ so that $\pp^{-1}Q\subset S.$ Note that, since $\pp\neq\qq_j$ for any $j,$ the prime factorization of $Q$ must not contain $\pp.$ Hence, $Q$ has some element, say $q,$ with zero $\pp$-adic valuation. But $\pp$ is a prime, hence maximal ideal of $\OO_K.$ Hence, the ring $\OO_K/\pp$ is a field, in particular the image of $q$ is invertible$\pmod\pp,$ say with inverse $q'.$ Thus, $qq'=1+p$ with $p\in\pp.$ But $$
\pp^{-1}qq'\subset\pp^{-1}q\subset\pp^{-1}Q\subset S,
$$
and $\pp^{-1}qq'=\pp^{-1}(1+p)\subset S.$ Since $p\in\pp,$ it has positive $\pp$-adic valuation, so the ideal $\pp^{-1}p\subset\OO_K\subset S.$ We then have $$
\pp^{-1}\subset\pp^{-1}(1+p)+\pp^{-1}p\subset S,
$$
as desired.
\end{proof}
\begin{corollary}\label{ring-in-terms-of-prime-valuation}
Consider any number field $K$ and an overring $S$ of $\OO_K$. Consider the set $\QQ$ of all prime ideals of $\OO_K$ that do not satisfy the condition of Lemma \ref{prime-inverse-in}, i.e. prime ideals $\qq$ such that any element of $S$ has nonnegative $\qq$-adic valuation. Then $S$ is exactly the set of elements $x\in K$ such that $v_\qq(x)\geq0$ for every prime ideal $\qq\in\QQ.$ 
\end{corollary}
\begin{proof}
Say $T$ is the set of elements $x\in K$ such that $v_\qq(x)\geq0$ for every prime ideal $\qq\in\QQ.$ Since $v_\qq(y)\geq0$ for any $y\in S$ and any $\qq\in\QQ,$ it must be true that $S\subset T.$ Conversely, take any element $s\in T,$ and consider the ideal factorization of $$
(s)=\qq_1^{\a_1}\qq_2^{\a_2}\cdots\qq_k^{\a_k}\pp_1^{-\b_1}\pp_2^{-\b_2}\cdots\pp_\ell^{-\b_\ell},
$$
with $\a_i,\b_j\geq0.$ By definition of $S,$ none of the $\pp_j$ are in $\QQ.$ Thus, they satisfy the condition of Lemma \ref{prime-inverse-in} so $\pp_j^{-1}\subset S.$ Since $S$ is a ring, $\pp_j^{-\b_j}\subset S$ as well. Additionally, $\qq_i^{\a_i}\subset\OO_K\subset S.$ But then we multiply these ideals to get that $$
\qq_1^{\a_1}\qq_2^{\a_2}\cdots\qq_k^{\a_k}\pp_1^{-\b_1}\pp_2^{-\b_2}\cdots\pp_\ell^{-\b_\ell}\subset S,
$$
so $(s)\subset S,$ so $s\in S.$ Hence, since $s\in T$ was chosen arbitrarily, $T\subset S,$ so $S=T,$ as desired.
\end{proof}

\subsection{Full description of \texorpdfstring{$R_+(\a)$}{} for algebraic integers}

From the results of the previous subsection it follows that to describe $R_+(\a)$ in terms of valuations it suffices to specify the set of prime ideals for which all elements of $R_+(\a)$ have a positive valuation. This is what we do in this subsection. However, we first need to develop a better understanding of the prime ideals in $\OO_K.$

For any number field $K,$ every prime ideal $\pp\subset\OO_K$ may be intersected with $\Z$ to form a prime ideal $p\Z$ of $\Z$ for some prime number $p.$ Thus, $p\in\pp,$ so the ideal $(p)=p\OO_K\subset\pp,$ therefore $\pp\mid(p).$ Consequently, we consider, for each prime $p,$ the prime ideals $\pp_i$ in the factorization of the ideal $(p)\in\OO_{\Q(\a)}.$

For any prime $q\neq p,$ the ideal $(q)$ does not have $\pp_i$ in its prime factorization. This is because otherwise $p,q\in\pp_i,$ so $\gcd(p,q)=1\in\pp_i,$ contradiction. In particular, in the element $\binom{\alpha}{k}=\frac{\alpha(\alpha-1)\cdots(\alpha-k+1)}{k!},$ we can immediately calculate the $\pp$-adic valuation of the denominator: $$v_\pp(k!)=v_\pp(p^{v_p(k!)})=v_p(k!)v_\pp(p)=v_\pp(p)\sum\limits_{n=1}^\infty\floor{\frac{k}{p^n}}.$$
Below we present well-known results regarding the theory of valuations and p-adic numbers, which we use to prove our main results.

Define the \textbf{p-adic rationals} $\Q_p$ as a completion of the rationals under the $p$-adic norm which, evaluated on a rational number $q\in\Q,$ equals $\frac{1}{p^{v_p(q)}}.$ It can be shown that $\Q_p$ is a field. Also, consider the subset of \textbf{p-adic integers} $\Z_p$ that is the completion of the integers $\Z.$ Then $\Z_p$ is an integrally closed ring with field of fractions $\Q_p$ (see \cite{milneANT}, Proposition 6.2). We are interested in the factorization of polynomials in $\Q_p[x].$ As it turns out, there is a surprising correspondence between the prime ideals of a ring of integers $\OO_{\Q(\a)}$ and the factorization of the minimal polynomial $m_\a$ over $\Q_p.$
\begin{prop}[McDonald \cite{primeidealvaluationstatements}, Neukirch \cite{alg-number-theory-neukirch}]\label{correspondence-between-p-adic-fields}
For any prime number $p\in\Z,$ the prime ideals in the factorization of the ideal $(p)\in\OO_{\Q(\a)}$ can be paired in one-to-one correspondence with the irreducible factors of $m_\a$ over $\Q_p[x].$ 
\end{prop}
Note that, since $m_\a$ is a minimal polynomial (hence is irreducible over $\Q),$ it must have a non-zero discriminant. Hence, the irreducible factors of $m_\a$ over $\Q_p$ are unique. 
If $m_\a$ is irreducible in $\Q_p[x],$ then the ideal $(p)$ has only one prime factor, hence is a power of a prime ideal.

Additionally, there exists a correspondence between $\pp$-adic valuations and the $p$-adic valuations in $\Q_p$ and its extensions. 
\begin{prop}[McDonald \cite{primeidealvaluationstatements}, Neukirch \cite{alg-number-theory-neukirch}]\label{v_P-to-v_p}
If the prime ideal $\pp_i\subset(p)$ is associated with the irreducible factor $g_i(x)$ of $m_\a$ in $\Q_p[x],$ then construct the field extension $\Q_p[x]/g_i(x)\simeq \Q_p(y)$ of $\Q_p,$ naturally extending the $p$-adic valuation to this larger field.\footnote{Define the $p$-adic valuation of an element of a field extension $K/\Q_p$ of degree $n$ as the $n$'th root of the $p$-adic valuation of its norm.} Then, for any polynomial $r(x)\in\Q[x],$ the equality $v_{\pp_i}(r(\a))=c\cdot v_p(r(y)),$ where $c=v_{\pp_i}(p),$ holds.
\end{prop}
\begin{proof}
This is a well-known result, but we provide an outline of the proof here.
From Proposition \ref{correspondence-between-p-adic-fields} we know (see the proof of Theorem 2.10 in \cite{primeidealvaluationstatements}) that there is $g_i(x)\in\Q_{p}[x]$ such that there is an isomorphism $\left(\underset{\leftarrow}{\lim}\OO_K /(\pp_i)^k\right)\left[\frac{1}{p}\right]\simeq \Q_p[x]/g_i(x)\simeq\Q_p(y)$, where $\Q_p(y)$ is an extension of the $p$-adic numbers. We know that this extension has a unique valuation $v_p$ which agrees with a standard valuation on $\Q_p$. Also, $v_{\pp_i}$ gives us a valuation on the inverse limit. But notice that $v_{\pp_i}(p)=c$ may not be equal to $1$.

Nevertheless, if we consider a normalized valuation $\frac{v_{\pp_i}}{c}$, we get a valuation on $\Q_p(y)$ that extends a standard valuation on $\Q_p$. Hence by uniqueness of such a valuation, if we denote by $\psi$ a map $$
\psi:\OO_K\rightarrow\left(\underset{\leftarrow}{\lim}\OO_K /(\pp_i)^k\right)\left[\frac{1}{p}\right] \simeq\Q_p[x]/g_i(x),
$$
then we must have $v_p(\psi(\gamma))=\frac{v_{\pp_i}(\g)}{c}$ for any $\g\in\OO_K,$ which is the statement of the proposition.
\end{proof}
This now allows us to calculate, albeit not exactly, the $\pp_i$-adic valuations of $\binom{\a}{k}.$ We split into two cases, when $g_i$ is linear and when it is not.
\begin{lemma}\label{linear-is-not-removed}
If $g_i$ is linear, then $v_{\pp_i}\left(\binom{\a}{k}\right)\geq0$ for all $k.$
\end{lemma}
\begin{proof}
The denominator of $\binom{\a}{k},$ as calculated above, has $\pp_i$-adic valuation $v_{\pp_i}(p)\sum\limits_{n=1}^\infty\floor{\frac{k}{p^n}}.$ Now, note that $\deg(g_i)=1,$ implying that both $e_i$ - the exponent of $\pp_i$ in the factorization of $(p)$ - and the inertial degree of $\pp_i$ (the $p$-adic valuation of the cardinality of $\OO_K/\pp_i)$ are equal to $1.$ Thus, the denominator has valuation $\sum\limits_{n=1}^\infty\floor{\frac{k}{p^n}}.$ It suffices to show that the numerator has at least this valuation.

Say $g_i=x-z_i,$ so $z_i$ is a root of $m_\a$ in $\Z_p.$ Note that, using the notation of Proposition \ref{v_P-to-v_p}, $y=z_i\in\Q_p,$ i.e. in this case we deal with $\Q_p$ itself and not an extension. Hence, we have $v_{\pp_i}(\a-s)=v_p(z_i-s).$ But note that $z_i$ is a $p$-adic integer, hence we can consider its value modulo $p^n.$ In particular, for any $s\equiv z_i\pmod{p^n},$ we have that $z_i-s$ is divisible by $p^n$ so $v_p(z_i-s)\geq n.$ Hence, $v_{\pp_i}(\a-s)\geq n.$ As before, among the integers $0,1,\dots,k-1$ there are at least $\floor{\frac{k}{p^n}}$ equivalent to $z_i\pmod{p^n},$ so these contribute an extra $\floor{\frac{k}{p^n}}$ to the $\pp_i$-adic valuation. Hence, in total, the $\pp_i$-adic valuation of the numerator $\a(\a-1)\cdots(\a-k+1)$ is at least $\sum\limits_{n=1}^\infty\floor{\frac{k}{p^n}},$ as desired.
\end{proof}
And now the case of non-linear factors in $m_\a,$ which requires some more insight. First, a lemma about $p$-adic numbers, which implies that we can project algebraic numbers over $\Q_p$ onto $\Q_p.$ Note that the Riesz Lemma does not apply, since $\|\cdot\|_p$ is not a norm in the usual sense.
\begin{lemma}\label{min-dist-Q_p}
For any $y$ algebraic over $\Q_p,$ there exists an element $q\in\Q_p$ for which the $p$-adic distance between $y$ and $q,$ equal to $\|y-q\|_p,$ is minimized, i.e. for any other $q' \in \Q_p$ we have $\|y-q'\|_p\geq\|y-q\|_p.$
\end{lemma}
\begin{proof}

First, note that the p-adic norm $\|\cdot\|_p$ is continuous in the topology defined by it. In particular, on any compact subset of $\Q_p,$ it achieves a minimum value. 

Fix some element $q\in\Q_p,$ and consider $\|y-q\|_p<\infty.$ Now, for any $q'$ with greater norm than $q$ and twice greater than that of $y-q,$ we see, noting that $\|a+b\|_p=\max(\|a\|_p,\|b\|_p)$ if $\|a\|_p\neq\|b\|_p,$ $$
\|y-q'\|_p\geq\|q-q'\|_p-\|y-q\|_p=\|q'\|_p-\|y-q\|_p>\|y-q\|_p.
$$
In particular, for some $n,$ for any $q'$ with $\|q'\|_p>p^n,$ we have that $\|y-q'\|_p>\|y-q\|_p$ for some $q$ with $\|q\|_p\leq p^n.$ Hence, we have $\inf\limits_{q\in\Q_p}\|y-q\|_p=\inf\limits_{q\in p^{-n}\Z_p}\|y-q\|_p.$ But, since $\Z_p$ is compact, $p^{-n}\Z_p$ is as well. Hence, there exists $q^*\in p^{-n}\Z_p$ such that $$
\|y-q^*\|_p=\inf\limits_{q\in p^{-n}\Z_p}\|y-q\|_p=\inf\limits_{q\in\Q_p}\|y-q\|_p,
$$
as desired.
%
%
\end{proof}
\begin{lemma}\label{nonlinear-is-removed}
If $g_i$ is nonlinear, then there exists $k$ such that $v_{\pp_i}\left(\binom{\a}{k}\right)<0.$
\end{lemma}
\begin{proof}
First, note that in this case the element $y,$ which generates the extension of $\Q_p$ given by $\Q_p[x]/g_i(x)$ (see the notation of Proposition \ref{v_P-to-v_p}), is not contained in $\Q_p,$ since the extension is non-trivial. Thus, by Lemma \ref{min-dist-Q_p}, there exists some element $q\in\Q_p$ that minimizes the $p$-adic distance from $y,$ which is $\|y-q\|_p>0$ (it is nonzero since $y\not\in\Q_p).$ Hence, $q$ maximizes $v_p(y-q)=m<\infty.$ Note that theoretically $m$ may be a rational number and not an integer.

Now, we wish for every integer $k$ to consider the numbers $z\in\Z$ such that $v_p(y-z)>k.$ This allows us to bound the $p$-adic valuation of the product $y(y-1)\cdots(y-p^n+1)$ for large $n.$ First, note that if $v_p(q-z)<v_p(y-q),$ then $$
v_p(y-z)=v_p((y-q)+(q-z))=v_p(q-z).
$$
On the other hand, if $v_p(q-z)\geq v_p(y-q),$ then $v_p(y-z)\geq v_p(y-q),$ but by definition of $q$ as the element of $\Q_p$ at which the maximum $p$-adic valuation is achieved, we must have $v_p(y-z)=v_p(y-q)=m.$

In particular, the $p$-adic valuation of $y-z$ is a discrete function, which can only equal either integers strictly less than $m$ or $m$ itself. For any integer $n<m,$ we have $v_p(y-z)=n$ iff $v_p(q-z)=n,$ so $z\equiv q\pmod{p^n}.$ And $v_p(y-z)=m$ iff $z\equiv q\pmod{p^{\lceil m\rceil}}.$ Thus, we have $$
v_p(y(y-1)\cdots(y-p^{\lceil m\rceil+1}+1))=\sum_{n=1}^{\lceil m\rceil-1}p^{\lceil m\rceil-n+1}+(m+1-\lceil m\rceil)p<\sum_{n=1}^\infty\floor{\frac{p^{\lceil m\rceil+1}}{p^n}}.
$$
Hence, by Proposition \ref{v_P-to-v_p}, the numerator of $\binom{\a}{p^{\lceil m\rceil+1}}$ has $\pp_i$-adic valuation strictly less than $v_{\pp_i}(p)\sum\limits_{n=1}^\infty\floor{\frac{p^{\lceil m\rceil}+1}{p^n}}.$ On the other hand, the denominator, as before, has $\pp_i$-adic valuation $v_{\pp_i}(p)\sum\limits_{n=1}^\infty\floor{\frac{p^{\lceil m\rceil}+1}{p^n}}.$ Thus, the $\pp_i$-adic valuation of $\binom{\a}{p^{\lceil m\rceil+1}}$ is strictly negative, as desired.
\end{proof}
Lemmas \ref{nonlinear-is-removed} and \ref{linear-is-not-removed} allow us to form a complete characterization of $R_+(\a)$ based on Corollary \ref{ring-in-terms-of-prime-valuation}. We first need a definition, though.
\begin{definition}\label{Q_p}
For an algebraic integer $\a,$ for every prime $p,$ define $\PP_p$ to be the (finite) set of prime ideals $\pp$ of $\OO_{\Q(\a)}$ over $(p)$ corresponding to linear factors of $m_\a$ in $\Q_p.$
Also consider the set $\{\qq_i\}$ of prime ideals over $(p)$ corresponding to nonlinear factors of $m_\a$. Define $Q_p=\prod_i\qq_i$, and let $Q_p=(1)=\OO_{\Q(\a)}$ if there are no such prime ideals.
\end{definition}
\begin{theorem}[Valuation Description]\label{valuation-description-for-alg-integers}
For any algebraic integer $\a$ that is not a nonnegative integer, the ring $R_+(\a)$ is equal to the set of all elements of $\Q(\a)$ with nonnegative $\pp$-adic valuation for every prime ideal $\pp\in\PP_p$ and prime number $p\in\Z.$
\end{theorem}
\begin{proof}
From Lemma \ref{linear-is-not-removed} it follows that all elements of $R_+(\a)$ have a nonnegative valuation with respect to any $\qq\in\PP_p$ and from Lemma \ref{nonlinear-is-removed} it follows that there are elements of $R_+(\a)$ which have negative valuation for all other prime ideals. Thus, Corollary \ref{ring-in-terms-of-prime-valuation} provides us with the desired result.
\end{proof}
We can also straightforwardly reformulate this Theorem in terms of the inverses of certain ideals from Definition \ref{Q_p}.
\begin{theorem}[Alternate Description]\label{exact-ideal-representation}
For any algebraic integer $\a$ that is not a nonnegative integer, the ring $R_+(\a)$ can be described as $R_+(\a)=\OO_{\Q(\a)}[Q_2^{-1},Q_3^{-1},\dots].$
\end{theorem}
\subsection{Extension to all algebraic numbers}
In the previous subsection we obtained a complete description of $R_+(\a)$ for algebraic integers. Here we extend this description to the case of general algebraic numbers.

First, we would like to provide a definition that allows us to connect the two cases of algebraic integers and algebraic non-integers.
\begin{definition}
For any algebraic number $\a,$ define the {\it denominator} of $\a$ as the smallest positive integer $\d$ such that $\b=\d\a$ is an algebraic integer. Define the {\it integral part} of $\a$ as the resulting value $\b.$
\end{definition}
It is well-known that we can multiply any algebraic number by a positive integer to obtain an algebraic integer, so the above notion is well-defined.

Now it seems straightforward that there should be some connection between $R_+(\alpha)$ and $R_+(\beta)$, where $\beta$ is the integral part of $\alpha$. However, the issue is that the added denominator may decrease the $\pp$-adic valuation of binomial coefficients for ideals $\pp$ corresponding to linear terms. This does not turn out to be a significant issue. Say $\a=\frac{\b}{\d},$ where $\b$ is the integral part and $\d$ is the denominator. As before, we care about the $\pp$-adic valuation of elements $\binom{\a}{k}.$ Since the $\pp$-adic valuation of any integer not divisible by $p$ is zero, we only care about the factors of $p$ in the denominator $\d.$ Now, we have $\binom{\a}{k}=\frac{\b(\b-\d)\cdots(\b-k\d+\d)}{\d^kk!}.$ But, if $\b$ has $\pp$-adic valuation greater than or equal to that of $\d,$ then all of the elements in the numerator automatically have $\pp$-adic valuation greater than or equal to that of $\d,$ which takes care of the $\d^k$ term in the denominator. It remains to examine the $k!$ part. Below we treat this case in more detail:
\begin{lemma}\label{ideal-inverse-in-iff-valuation-negative}
If $\pp$ corresponds to a linear factor of $m_\b$ in $\Z_p$ and $\a=\frac{\b}{\d}$ has nonnegative $\pp$-adic valuation, then for every integer $k$ the element $\binom{\a}{k}$ also has nonnegative $\pp$-adic valuation.
\end{lemma}
\begin{proof}
Say $v_p(\d)=e,$ so that $\d=p^ed$ where $d$ is not divisible by $p.$ As in Lemma \ref{linear-is-not-removed}, the denominator of $\binom{\a}{k}=\frac{\b(\b-p^ed)\cdots(\b-(k-1)p^ed)}{p^{ek}d^kk!}$ has $\pp$-adic valuation $$
v_\pp(p^{ek}d^kk!)=ek+\sum\limits_{n=1}^\infty\floor{\frac{k}{p^n}}.
$$
By Proposition \ref{v_P-to-v_p}, we have $v_\pp(\b-np^ed)=v_p(z-np^ed),$ where $z\in\Z_p$ is the root of $m_\beta$ corresponding to the ideal $\pp.$ Since $\b$ has greater $\pp$-valuation than $p^ed=\d,$ this means that $v_p(z)=v_\pp(\b)\geq v_\pp(p^ed)=e.$ Since $d$ is not divisible by $p,$ this means that for every $f>e,$ there exists a unique residue $z^{(f)}\pmod{p^{f-e}}$ such that $z^{(f)}p^ed\equiv z\pmod{p^f}.$ In particular, every number $n$ from $0$ to $k-1$ contributes $e$ to the total $\pp$-adic valuation of the product $\prod\limits_{n=0}^{k-1}(\b-np^ed),$ and for every $f>e,$ every number $n$ equivalent to $z^{(f)}\pmod{p^{f-e}}$ contributes an extra $1$ to the total valuation. As before, among the integers $0,1,\dots,k-1$ there are at least $\floor{\frac{k}{p^{f-e}}}$ that are equivalent to $z^{(f)}\pmod{p^{f-e}}.$ Hence, we have $$
v_\pp(\b(\b-p^ed)\cdots(\b-(k-1)p^ed))\geq ek+\sum\limits_{f=e+1}^\infty\floor{\frac{k}{p^{f-e}}}=ek+\sum\limits_{n=1}^\infty\floor{\frac{k}{p^n}},
$$
which is the $\pp$-adic valuation of the denominator. Hence, $\binom{\a}{k}$ has nonnegative $\pp$-adic valuation, as desired.
\end{proof}
In particular, the additional prime ideals whose inverses are appended are exactly the primes with respect to which $\a$ has a negative valuation. 
\begin{definition}\label{P_p-updated}
For an algebraic number $\a$ with integral part $\b,$ for every prime $p,$ define $\PP_p'$ to be the (finite) set of prime ideals $\pp$ of $\OO_{\Q(\b)}$ over $(p),$ corresponding to linear factors of $m_\b$ in $\Q_p,$ which satisfy $v_\pp(\a)\geq0.$ Denote this set by $\PP_p'$.
\end{definition}
\begin{theorem}[Valuation Description]\label{exact-description-for-all-alg-nums-valuation}
For any algebraic number $\a$ that is not a nonnegative integer, the ring $R_+(\a)$ is equal to the set of all elements of $\Q(\a)$ with nonnegative $\pp$-adic valuation for every prime ideal $\pp\in\PP_p'$ for every prime number $p\in\Z.$
\end{theorem}
\begin{proof}
Say $\b$ is the integral part of $\a$ and $\d$ is the denominator of $\a.$ Then from Lemma \ref{nonlinear-is-removed} we already know that there are elements of $R_+(\b)$ with negative valuation with respect to the prime ideals corresponding to nonlinear factors in $m_\b.$ Since $R_+(\b)\subset R_+(\a)$ it follows that there are such elements in $R_+(\a)$. Now, $R_+(\a)$ also contains elements with negative $\pp$-adic valuation for every prime $\pp$ such that $v_\pp(\a)<0$ (since $\a\in R_+(\a)).$ 
Lemma \ref{ideal-inverse-in-iff-valuation-negative} states that all elements of $R_+(\a)$ have nonnegative valuation with respect to all other prime ideals. Hence, the result follows from Corollary \ref{ring-in-terms-of-prime-valuation}. Note that $m_\a(x)=m_\b(\d x),$ so $m_\a$ factorizes in exactly the same way as $m_\b.$ Therefore, the same prime ideals correspond to factors of $m_\a,$ as desired.
\end{proof}
We can provide another description of $R_+(\a)$.
Indeed, from Lemma \ref{prime-inverse-in}, we know that $\OO_{\Q(\a)}[\a]$ must contain the inverses of all prime ideals such that $v_{\pp}(\a)<0$. In particular, we obtain the following result.
\begin{prop}\label{alg-num-to-alg-int}
For any algebraic number $\a$ that is not a nonnegative integer, the ring $R_+(\a)$ is equal to $R_+(\OO_{\Q(\a)})[\a].$
\end{prop}
\begin{proof}
The inclusion of $R_+(\OO_{\Q(\a)})[\a]$ in $R_+(\a)$ is obvious. Conversely, if we write $\a=\frac{\b}{\d},$ where $\d$ is the denominator of $\a,$ then note that $R_+(\OO_{\Q(\a)})=R_+(\b).$ But $R_+(\b)$ contains $\OO_{\Q(\a)},$ and so $R_+(\b)[\a]$ contains the inverses of all prime ideals $\pp$ with respect to which $\a$ has negative $\pp$-valuation. But this means, by Lemma \ref{ideal-inverse-in-iff-valuation-negative} in conjunction with Corollary \ref{ring-in-terms-of-prime-valuation}, that $R_+(\b)[\a]$ contains $R_+(\a),$ as desired.
\end{proof}
We can also unpack this result a bit more using Theorem \ref{exact-ideal-representation}.
\begin{theorem}[Alternative Description]\label{exact-description-for-all-alg-nums}
For any algebraic number $\a$ with integral part $\b,$ consider $Q_p$ as defined with respect to $\b$ in Definition \ref{Q_p}. Then $$
R_+(\a)=R_+(\b)[\a]=R_+(\OO_{\Q(\a)})[\a]=\OO_{\Q(\a)}[\a,Q_2^{-1},Q_3^{-1},\dots].
$$
\end{theorem}

We end this section by providing a proof of Proposition \ref{R+-of-ring-is-a-ring}, described at the end of Section \ref{section-integral-elements}. We actually prove a stronger statement, describing the exact structure of $R_+(S)$ for any ring $S.$
\begin{lemma}\label{R+-of-ring-is-ring}
For any ring $S$ with field of fractions $K$ algebraic over $\Q,$ the set $R_+(S)$ is a ring equal to $R_+(\OO_K)[S].$
\end{lemma}
\begin{proof}
First, any element of $K$ can be expressed as a finite linear combination with rational coefficients of elements of $S.$ Now, take any element $k\in\OO_K.$ Then there exists a positive integer $n$ so that the expression of $nk$ as a linear combination of elements of $S$ has all integral coefficients (multiply through by the common denominator of the rational coefficients). Hence, $nk\in S.$ But then $$
k\in\OO_{\Q(nk)}\subset R_+(nk)\subset R_+(S).
$$
Now, $k$ was chosen as an arbitrary element of $\OO_K,$ so $\OO_K\subset R_+(S),$ as desired.

By Proposition \ref{alg-num-to-alg-int}, for any $s\in S,$ $R_+(s)=R_+(\OO_{\Q(s)})[s]\subset R_+(\OO_K)[S].$ Hence, $R_+(S)\subset R_+(\OO_K)[S]$ (since it consists of linear combinations of elements of $R_+(s)).$ It remains to show that $R_+(S)$ is a ring, as then it must contain $R_+(\OO_K)[S].$ 

First, since $S$ is a ring, $-1\in S\subset R_+(S).$ Hence, $R_+(S)$ is closed under subtraction. Thus, it is sufficient to show that the product of any two binomial coefficients of elements of $S,$ $\binom{r}{k}\binom{s}{\ell}\in R_+(S)$ (since the product of any two elements in $R_+(S)$ is expressible as a sum of products of binomial coefficients). Since $\binom{r}{k}\in R_+(r)$ and $\binom{\a}{k}\in R_+(s),$ it suffices to show that the product of any two elements from $R_+(r)$ and $R_+(s)$ is contained in $R_+(S).$

Note that $R_+(r)=R_+(\OO_{\Q(r)})[r]$ and $R_+(s)=R_+(\OO_{\Q(s)})[s].$ In particular, any element of $R_+(r)$ is of the form $a_0+a_1r+\dots+a_nr^n,$ where $a_i\in R_+(\OO_{\Q(r)}),$ and any element of $R_+(s)$ is of the form $b_0+b_1s+\dots+b_ms^m,$ where $b_j\in R_+(\OO_{\Q(s)}).$ But then the product of these two elements is of the form $\sum\limits_{i=0}^n\sum\limits_{j=0}^ma_ib_jr^is^j.$ It thus suffices to show that $a_ib_jr^is^j\in R_+(S).$ Set $r^is^j=t\in\Z[r,s]\subset S,$ $a_ib_j=c\in R_+(\OO_{\Q(r,s)}).$ Take a generator (primitive element) $u$ of the field $\Q(r,s).$ Then consider an integer $n$ such that $nu\in\Z[r,s]\subset S$ and an integer $m$ such that $mu\in\OO_{\Q(r,s)}$ (these are sure to exist). Then note that $mnu\in S\cap\OO_{\Q(r,s)}.$ Set $mnu=v.$ Then there exists some element of the form $pv+t,$ $p\in\Z,$ that generates all of $\Q(r,s)$ (otherwise, $\Q(r,s)$ has infinitely many subfields). Note that $pv+t\in S$ (since both $v$ and $t$ are) and $pv\in\OO_{\Q(r,s)}.$ But then $$
R_+(pv+t)=R_+(\OO_{\Q(r,s)})[pv+t]=R_+(\OO_{\Q(r,s)})[t].
$$
Hence, $ct\in R_+(pv+t)\subset R_+(S),$ as desired.
\end{proof}


\subsection{The structure of \texorpdfstring{$R_+(\a)$}{} and Deligne categories over p-adic numbers} \label{p-adic-functor-section}

In this section we show how the description of $R_+(\a)$ obtained in Theorem \ref{exact-description-for-all-alg-nums-valuation} relates to the theory of Deligne categories in positive characteristic. More explicitly, we show that there should be a connection between the existence of a symmetric monoidal functor between Deligne categories over $\mathbb C$ and the existence of a symmetric monoidal functor between Deligne categories in positive characteristic. The results of this section explain why the description of $R_+(\a)$ we obtained is natural and in fact is to be expected.

We begin by reformulating Theorem \ref{exact-description-for-all-alg-nums-valuation} in a slightly different form. In order to do this we first need to formulate a well-known number-theoretic result.
\begin{prop} [Section 8 of Chapter 2 of \cite{alg-number-theory-neukirch}] \label{homomorphism-prop}
For any algebraic number $\alpha$ and prime number $p,$ consider the corresponding ring of integers $\OO_{\Q(\a)}$ and its prime ideals over $p.$ There is a one-to-one correspondence between these ideals and field embeddings $\Q(\a) \to \overline{\Q_p}$. Moreover, in this correspondence, prime ideals $\mathfrak{p}$ with ramification index $1$ and such that $\OO_{\Q(\a)}/\mathfrak{p} = \mathbb F_p$ correspond to maps $\Q(\a) \to \overline{\Q_p}$ that factor through $\Q_p$.
\end{prop}
\begin{proof}
The full proof can be found in \cite{alg-number-theory-neukirch}. Here we only wish to note that the embedding corresponding to the ideal $\mathfrak{p}$ over $p$ has the following form. Denote by $g(x)$ the irreducible factor of the minimal polynomial of the integral part of $\alpha$ over $\Q_p$ that corresponds to $\mathfrak{p}$ (see Proposition \ref{correspondence-between-p-adic-fields}), then the corresponding map $\phi_{\mathfrak{p}}$ is the extension of the map
$$
 \OO_{\Q(\a)} \to \left(\underset{\leftarrow}{\lim}\OO_{\Q(\a)} /(\pp)^k\right)\left[\frac{1}{p}\right] \simeq\Q_p[x]/g(x) \to \overline{\Q_p} .
$$
to $\Q(\a)$. From this construction it is straightforward to see that in the case of linear $g$ (or equivalently $\mathfrak{p}$ of ramification index $1$ and with quotient field of cardinality $p$)  $\phi_{\mathfrak{p}}$ factors through $\Q_p$.
\end{proof}
If $\mathfrak{p}$ is a prime ideal of $\mathcal O_{\Q(\a)},$ we denote the corresponding field embedding by $\phi_{\mathfrak{p}}$. Also note that from Proposition \ref{v_P-to-v_p} we know that $v_{\mathfrak{p}} = v_{\mathfrak{p}}(p) \cdot (v_p \circ \phi_{\mathfrak{p}})$. With this we can reformulate our main result in the following way.
\begin{theorem} \label{another-alternative-theorem}
For any algebraic number $\alpha$ which is not a nonnegative integer, the subring $R_+(\alpha) \subset \Q(\a)$ is given by the elements $\g\in \Q(\a)$ such that, for any prime number $p$ and for any field embedding $\phi:\Q(\a)\to \Q_p$ such that $v_p(\phi(\alpha))\ge 0,$ we have $v_p(\phi(\g)) \ge 0$. 
\end{theorem}
\begin{proof}
Indeed, from Proposition \ref{homomorphism-prop} we know that, for any prime number $p,$ embeddings $\Q(\a)\to \Q_p$ correspond to prime ideals $\mathfrak{p}$ of $\OO_{\Q(\a)}$ such that the corresponding irreducible factor of $m_\beta$ ($\b$ is the integral part of $\a$) is linear. Now, the condition $v_p(\phi_{\mathfrak{p}}(\a))\ge 0$ is the same as the condition $v_{\mathfrak{p}}(\a)\ge 0$ for the corresponding ideal. Hence we can see that the set of all such field embeddings is in one-to-one correspondence with the set of prime ideals $\mathcal P_p'$ from Definition \ref{P_p-updated}. But then the condition on $\g$ from the statement of the Theorem is exactly the same as the condition that $v_{\mathfrak{p}}(\g)\ge 0$ for all $\mathfrak{p} \in \mathcal P_p'$. This result is therefore equivalent to Theorem \ref{exact-description-for-all-alg-nums-valuation}.
\end{proof}
We now apply this result to the study of Deligne categories. Namely, we discuss symmetric monoidal functors between Deligne categories in positive characteristic. In one of his letters to Ostrik, Deligne delineated how through a certain ultrafilter construction one can obtain a class of Deligne categories such that the ``rank'' of these categories turns out to be a $p$-adic integer. We will denote such a category by $\Rep_p(S_t)$, where $t \in \Z_p$, for our present discussion. A more detailed discussion of how these categories can be constructed is to be found in \cite{harman2015stability,etingof2018p}. However the knowledge of the properties of these categories is still quite limited. Hence the discussion below has a provisional character.

\begin{remark} \label{remark-p-char-deligne}
Suppose we have $\a,\g \in \overline{\Q}$ such that $\g \in R_+(\a),$ and so we have a symmetric monoidal functor from $\Rep(S_\g) $ to $\Rep(S_\a)$. Now, such a functor is given by an object $V$ of dimension $\g$ in $\Rep(S_\a)$ with the structure of a commutative Frobenius algebra\footnote{Here and below all Frobenius algebras are required to have a non-degenerate trace form.}; from \cite{classify-symmetric-monoidal-functors} we know that this object is given by a direct sum of simple commutative Frobenius algebras that interpolate algebras $\Ind_{S_{N-k}\times H}^{S_N}(\mathbb C)$ for certain values of $k$ and $H \subset S_k$. 

Now say that $\phi$ is a field embedding $\phi:\Q(\a) \to \Q_p$ for some prime number $p$ such that $\phi(\a) \in \Z_p$. It follows that the category $\Rep_p(S_{\phi(\a)})$ exists. Hence we can construct an object $V_p \in \Rep_p(S_{\phi(\a)})$ as a direct sum of the same simple commutative Frobenius algebras that constitute $V,$ reduced to characteristic $p.$\footnote{In other words, if one of the simple commutative Frobenius algebras constituting $V$ interpolates $\Ind_{S_{N-k}\times H}^{S_N}(\mathbb C)$, then the corresponding simple algebra in $V_p$ interpolates $\Ind_{S_{N-k}\times H}^{S_N}(\mathbb F_p)$.}

Now, in the setting of symmetric tensor categories in characteristic $p,$ one can define the notion of $p$-adic dimension, which we will here denote as $\Dim.$ The usual categorical dimension, which in this case takes values in $\F_p$, gives us the first $p$-adic digit of $\Dim.$\footnote{The other digits can be computed via the categorical dimensions of the exterior powers of the given object. See \cite{etingof2018p} for details.} It is conjectured that in  characteristic $p$ there is an analogue of the universal property for Deligne categories that would state that  symmetric monoidal functors from the Deligne category $\Rep_p(S_t)$ to another symmetric tensor category $\mathcal T$ in characteristic $p$, under certain conditions, correspond bijectively to commutative Frobenius algebras in $\mathcal T$ with $p$-adic dimension $t$.

Assuming the truth of this conjecture, we can conclude that the object $V_p \in \Rep_p(S_{\phi(\a)})$ gives us a symmetric monoidal functor $\Rep_p(S_{\Dim(V_p)}) \to \Rep_p(S_{\phi(\a)})$. But then, since $\phi$ is a field embedding, and $V_p$ is a reduction of $V$ in characteristic $p$, we conclude that the dimension of $V_p$ has to be $\Dim(V_p) = \phi(\g) \in \Z_p$. 

Thus, the fact that the $p$-adic valuation $v_p(\phi(\g)) \ge 0$ for any prime number $p$ and field embedding $\phi:\Q(\a) \to \Q_p$ such that $v_p(\phi(\a))\ge 0$ can be seen as a natural consequence of the fact that a functor between Deligne categories in zero characteristic can be reduced to a functor between Deligne categories in characteristic $p$. 
\end{remark}

\section{Explicit generators for \texorpdfstring{$R_+(\alpha)$}{}}\label{Specifics-Section}
In this section, we show a formula for $R_+(\a)$ as a ring generated by certain, explicitly defined algebraic elements. We assume throughout this section that $R_+(\a)$ is a ring, i.e. $\a$ is not a nonnegative integer.
\subsection{Construction of \texorpdfstring{$\alpha_p$}{}}
Here, we use roots of $m_\a$ in $\Z_p$ to construct an element of $R_+(\a),$ and we examine the scenarios when this element is equal to $\frac{1}{p}.$ We assume that $\a$ is an algebraic integer. For the case of a non-integral algebraic number it is enough to consider its integral part by virtue of Proposition \ref{alg-num-to-alg-int}.

Our goal is to construct, for every prime $p,$ an element $\a_p$ which has a valuation of zero modulo every prime ideal corresponding to a root of $m_\a$ in $\Z_p,$ and a negative valuation modulo every prime ideal corresponding to an irreducible factor of degree $>1.$ Then, we would like to consider the ideal $(\a_p),$ which contains $Q_p^{-1}$ in its factorization, but no negative factors of ideals from $\PP_p$.
\begin{definition}\label{a_p}
For any algebraic integer $\a$ and for any prime $p,$ consider the set of (distinct) roots $z_1,z_2,\dots z_q$ of $m_\a$ in $\Z_p,$ and consider the smallest positive integer $k$ such that they are all distinct modulo $p^k$ (such a $k$ exists because the $z_i$ are distinct). For each $i,$ define $k_i=\sum\limits_{j\neq i}v_p(z_j-z_i),$ and define $K=\max\limits_ik_i.$ For each $i,$ let $m_i=K-k_i+k.$ Finally, by Lemma \ref{min-dist-Q_p}, the number $\max\limits_{z\in\Z_p}v_p(z-y)$ exists and is achievable for any number $y$ which is algebraic over $\Q_p.$ Therefore, we can consider the integer $N=\lceil\max\limits_j\max\limits_{z\in\Z_p}v_p(z-y_j)\rceil,$ where $y_j$ are the roots over $\Z_p$ of the irreducible factors of $m_\a.$ Then for each $i$ choose an integer $r_i$ such that $z_i$ is equivalent to $r_i$ modulo $p^{m_i+qN}$ but not modulo $p^{m_i+qN+1}.$ We then define $$
\a_p=\frac{(\a-r_1)(\a-r_2)\cdots(\a-r_q)}{p^{K+k+qN}}.
$$
\end{definition}
Note that $\a_p$ is a modified version of $\binom{\a}{p^\ell}$ for some $\ell,$ with a denominator consisting only of the factors of $p$ and some factors $(\a-i)$ removed. Also, if $m_\a$ has no roots in $\Z_p,$ then $k=1$ and $\a_p=\frac{1}{p}.$
\begin{lemma}\label{inclusion-of-a_p}
For any prime $p,$ the element $\a_p\in R_+(\a).$
\end{lemma}
\begin{proof}
We use the valuation version of the main result of Section \ref{General-Section}, i.e. Theorem \ref{exact-description-for-all-alg-nums-valuation}. First, the only ideals with possibly negative valuation of $\a_p$ are those in the factorization of $(p)$ (since the denominator of $\a_p$ is a power of $p).$ Hence, we only need to consider the ideals in the factorization of $(p),$ and moreover only those corresponding to linear factors of $m_\a$ in $\Z_p,$ as the other ideals' inverses are contained in $R_+(\a).$ In other words, we consider ideals in $\PP_p.$

Say $\pp_i$ is the ideal corresponding to the root $r_i\in\Z_p.$ We wish to show that $v_{\pp_i}(\a_p)\geq0.$ Now, we have $$
v_{\pp_i}(\a-r_i)=v_p(z_i-r_i)=m_i+qN
$$
since $z_i-r_i$ is divisible by $p^{m_i+qN}$ but not by $p^{m_i+qN+1}.$ In addition, $$
v_{\pp_i}(\a-r_j)=v_p(z_i-r_j)=v_p(z_i-z_j)
$$
because $v_p(r_j-z_j)>v_p(z_i-r_j),$ since $z_i$ are distinct modulo $p^k.$ Therefore, we have $$
v_{\pp_i}(\a_p)=v_{\pp_i}(\a-r_i)+\sum_{j\neq i}v_{\pp_i}(\a-r_j)-v_{\pp_i}(p^{K+k+qN})=m_i+qN+k_i-(K+k+qN)=0
$$
by definition of $m_i$. Hence, since $\pp_i$ was chosen arbitrarily, $v_\pp(\a_p)=0$ for all $\pp\in\PP_p,$ so by Theorem \ref{exact-description-for-all-alg-nums-valuation}, $\a_p\in R_+(\a),$ as desired.
\end{proof}
Also note that in the case when $m_\a$ has no roots on $\Z_p$, $\PP_p$ is empty, so we have nothing to check and $\alpha_p =\frac{1}{p}$ automatically lies in $R_+(\a)$.
In particular, we obtain a strong result about the inclusion of the inverses of many primes in $R_+(\a).$
\begin{corollary}\label{inverses-of-primes-in}
For any algebraic number $\a,$ if $m_\a(x)$ has no roots in $\Z_p,$ then $\frac{1}{p}\in R_+(\a).$ 
\end{corollary}
\begin{proof}
This immediately follows from Lemma \ref{inclusion-of-a_p} by noting that $\a_p=\frac{1}{p}.$
\end{proof}
In fact, Corollary \ref{inverses-of-primes-in} proves the inclusion of an infinitude of inverses of primes in $R_+(\a)$ of nonzero density in the set of all primes, by Chebotarev's Density Theorem \cite{chebotarev-density}.
\subsection{Description of \texorpdfstring{$R_+(\alpha)$}{} in terms of \texorpdfstring{$\alpha_p$}{}}
In the previous subsection, we defined an element $\a_p$ and proved that it is contained in $R_+(\a).$ Here, we prove a ``converse": that $\OO_{\Q(\a)}[\a_p]$ encompasses almost all of $R_+(\a).$ This leads to an explicit description of $R_+(\a)$ in terms of generators.
\begin{lemma}\label{negative-valuation-of-nonlinear-ideals}
For any prime ideal $\pp$ corresponding to an irreducible nonlinear factor of $m_\a$ in $\Z_p,$ the valuation $v_\pp(\a_p)<0.$
\end{lemma}
\begin{proof}
Say $y$ is the generator of the algebraic extension of $\Q_p$ given by the irreducible nonlinear factor. Then we have, where $c=v_{\pp}(p)$ is a positive integer, \begin{align*}
v_\pp(\a_p)=\sum_{j=1}^qv_\pp(\a-r_j)-v_\pp(p^{K+k+qN})&=c\left(\sum_{j=1}^qv_p(y-r_j)-(K+k+qN)\right)\\
&\leq c\left(\sum_{j=1}^qN-(K+k+qN)\right)=-c(K+k)<0,
\end{align*}
since $k>0,$ as desired.
\end{proof}

This allows us to see that any element of $R_+(\a)$ can be modified slightly to become an element of $\OO_{\Q(\a)}[\a_p].$
\begin{lemma}\label{multiply-into-a_p}
For any $\b\in R_+(\a),$ there exists an integer $n$ not divisible by $p$ such that $n\b\in\OO_{\Q(\a)}[\a_p].$
\end{lemma}
\begin{proof}
This is simply an application of Corollary \ref{ring-in-terms-of-prime-valuation}. Note that $\OO_{\Q(\a)}[\a_p]$ is a ring containing $\OO_{\Q(\a)},$ so we can combine Lemmas \ref{inclusion-of-a_p} and \ref{negative-valuation-of-nonlinear-ideals} to obtain that $\OO_{\Q(\a)}[\a_p]$ is exactly the set of all elements of $\Q(\a)$ with nonnegative $\qq$-adic valuation for any $\qq\in\PP_p$ or $\qq$ that divide $(q)$ for primes $q\neq p.$ Say $\b\in R_+(\a).$ Then we know, by Theorem \ref{valuation-description-for-alg-integers}, that $v_\pp(\b)\geq0$ for any $\pp\in\PP_p.$ Say $\d$ is the denominator of $\b,$ and then consider the element $\g=\left(\frac{\d}{p^{v_p(\d)}}\right)\b.$ For all prime ideals $\qq$ which are not in the prime factorization of $(p),$ the $\qq$-adic valuation of $\g$ is nonnegative because $\g$ is equal to an algebraic integer $(\d\b)$ divided by a power of $p.$ In addition, for any prime ideals $\pp\in\PP_p,$ $v_\pp(\b)\geq0$ so $v_\pp(\g)\geq0$ as well. Hence, $\g\in\OO_{\Q(\a)}[\a_p].$ But $\left(\frac{\d}{p^{v_p(\d)}}\right)$ is an integer not divisible by $p,$ as desired.
\end{proof}
This leads us to the main, explicit theorem.
\begin{theorem}\label{final-description}
For $\a$ an algebraic integer that is not a nonnegative integer, the ring $$
R_+(\a)=\OO_{\Q(\a)}\left[\a_2,\a_3,\a_5,\a_7,\dots\right].
$$
\end{theorem}
\begin{proof}
The inclusion $\OO_{\Q(\a)}\left[\a_2,\a_3,\dots\right]\subset R_+(\a)$ follows immediately from Lemma \ref{inclusion-of-a_p}. For the opposite inclusion, suppose $\b\in R_+(\a),$ and consider the denominator $\d$ of $\b.$ In particular, $\d\b\in\OO_{\Q(\a)}.$ Fix any prime $p$ dividing the denominator $\d.$ By Lemma \ref{multiply-into-a_p}, there exists some $n_p$ not divisible by $p$ such that $$
n_p\b\in\OO_{\Q(\a)}[\a_p]\subset\OO_{\Q(\a)}\left[\a_2,\a_3,\dots\right]=S.
$$
Then consider the set $T:=\{\d\}\cup\{n_p:p\mid\d\}.$ Note that the set of elements $T\b\subset S.$ None of the primes dividing $\d$ divide all elements of $T,$ hence the set $T$ has greatest common divisor $1.$ Thus, by Bezout's Lemma, some integral linear combination of the elements of $T$ equals $1.$ But then $\b\in S,$ as desired.
\end{proof}
For algebraic numbers $\a,$ we just need to append $\a$: 
\begin{theorem}\label{final-description-all-nums}
For $\a$ an algebraic number with integral part $\b,$ the ring $$
R_+(\a)=\OO_{\Q(\a)}\left[\a,\b_2,\b_3,\b_5,\dots\right].
$$
\end{theorem}
\begin{proof}
This follows directly from Theorem \ref{final-description} and Proposition \ref{alg-num-to-alg-int}.
\end{proof}
These two results provide another description of $R_+(\a)$ which might prove more useful in certain applications.

\subsection{The case of prime numbers not dividing the discriminant}
Although we have a general form, it may be tedious to find $\a_p$ for $\a$ an algebraic integer, since this requires a calculation of $p$-adic valuations of algebraic elements over $\Z_p$ (which is relatively difficult). However, in a special case, which turns out to be very prevalent, the calculation can be simplified. Specifically, we consider the primes $p$ not dividing the discriminant of $m_\a.$ Then the factorization of $m_\a$ modulo $p$ extends uniquely modulo every power of $p$ \cite{henselslemma}. Therefore, instead of considering $m_\a$ in $\Z_p,$ we simply have to consider $m_\a$ modulo $p.$
\begin{lemma}\label{primes-not-dividing-discriminant}
Suppose $p\nmid\disc(m_\a).$ Then consider the roots $r_1,r_2,\dots,r_q$ of $m_\a$ modulo $p.$ Then Lemmas \ref{inclusion-of-a_p}, \ref{negative-valuation-of-nonlinear-ideals}, and \ref{multiply-into-a_p} apply for $\a_p=\frac{(\a-r_1)(\a-r_2)\cdots(\a-r_q)}{p}.$
\end{lemma}
\begin{proof}
Lemma \ref{multiply-into-a_p} follows from Lemmas \ref{inclusion-of-a_p} and \ref{negative-valuation-of-nonlinear-ideals}, and Lemma \ref{inclusion-of-a_p} is simple to check. Hence, it remains to demonstrate the validity of Lemma \ref{negative-valuation-of-nonlinear-ideals}.

We wish to show, for every ideal $\qq$ corresponding to an irreducible nonlinear factor of $m_\a$ in $\Z_p,$ that $v_\qq(\a_p)<0.$ Say $y$ is a root of $m_\a$ (algebraic over $\Q_p).$ Then consider the irreducible factor $g$ of $m_\a$ of which $y$ is a root. Suppose now that $v_p(y-d)=k>0$ for some $d\in\Z.$ Then $$
0=g(y)=g(d+(y-d))\equiv g(d)\pmod{p^k}.
$$
Since $m_\a$ is monic with integral coefficients $(\a$ is an algebraic integer), $y$ is an algebraic integer over $\Z_p,$ so $g,$ the minimal polynomial of $y,$ must also be monic with integral coefficients. Hence, $g(d)$ is an integer divisible by a nonzero power of $p,$ so it must be divisible by $p.$ But then $d$ is a root of $g$ in $\F_p,$ so by Hensel's lemma $(g$ is a factor of $m_\a,$ so its discriminant is not divisible by $p$ either), $d$ extends to a root $d'$ of $g$ in $\Z_p.$ But this is a contradiction, since $g$ was assumed to be irreducible over $\Q_p.$ Hence, $v_p(y-d)=0$ for every integer $d.$
Therefore, letting $c=v_\qq(p)$ be the scaling factor, we have $$
v_\qq(\a_p)=cv_p\left(\frac{(y-r_1)(y-r_2)\cdots(y-r_q)}{p}\right)=c\left(\sum\limits_{i=1}^qv_p(y-r_i)-1\right)=-c<0,
$$
as desired.
%
\end{proof}
Since the discriminant of $m_\a$ is a nonzero integer $(m_\a$ is irreducible, hence separable, over $\Q),$ it is only divisible by finitely many primes. Hence, this significantly simplifies the calculations for all but finitely many primes.
\section{Examples}\label{Examples-Section}
In this section we study $R_+(\a)$ for specific types of algebraic numbers $\a.$
\subsection{Quadratic elements}\label{Quadratic-Section}
We first discuss, and completely solve, the case when $\a$ is an algebraic number of degree $2.$ These are numbers that are roots of quadratics $ax^2+bx+c,$ where $a,b,$ and $c$ are integers and the quadratic has no rational roots. Note that $\a=\frac{-b\pm\sqrt{b^2-4ac}}{2a}.$ We can simplify $b^2-4ac$ to remove the square factors to get that $\a=\frac{m+n\sqrt{d}}{p}$ for some integers $m,n,p,d$ with $d$ squarefree.
\begin{lemma}\label{quadratic-ring}
$R_+(\a)$ is a ring.
\end{lemma}
\begin{proof}
This follows from the fact that $\a$ is of degree $2,$ hence is an algebraic number that is not a nonnegative integer.
\end{proof}
By Proposition \ref{alg-num-to-alg-int}, we have $$
R_+(\a)=R_+(\OO_{\Q(\a)})[\a]=R_+(\OO_{\Q(\sqrt{d})})[\a]=R_+(\sqrt{d})[\a].
$$
In particular, 
it suffices to consider $R_+(\sqrt{d}).$
First, generalize the Legendre symbol, allowing a denominator of $2.$ Specifically, say $\left(\frac{d}{2}\right)=1$ if $d$ is a quadratic residue$\pmod{2^k}$ for any $k,$ which happens iff $d\equiv1\pmod8$ \cite{henselslemma}. Say $\left(\frac{d}{2}\right)=-1$ if $d$ is a quadratic residue$\pmod4$ but not$\pmod8,$ i.e. $d\equiv5\pmod8,$ $\left(\frac{d}{2}\right)=i,$ the imaginary unit, if $d$ is a quadratic residue$\pmod2$ but not$\pmod4,$ i.e. $d\equiv3\pmod4,$ and finally $\left(\frac{d}{2}\right)=0$ if $d$ is divisible by $2.$
\begin{lemma}\label{contains-inverses-of-primes}
For any prime $p,$ if $\left(\frac{d}{p}\right)\neq1,$ then $\frac{1}{p}\in R_+(\sqrt{d}).$
\end{lemma}
\begin{proof}
By definition, the polynomial $x^2-d$ has no roots in $\Z_p$ in this case, so by Corollary \ref{inverses-of-primes-in}, $(\sqrt{d})_p=\frac{1}{p}\in R_+(\sqrt{d}),$ as desired.
\end{proof}
In particular, we have $$
\OO_{\Q(\sqrt{d})}\left[\frac{1}{p_1},\frac{1}{p_2},\dots\right]\subseteq R_+(\sqrt{d}),
$$
where $p_i$ are the primes such that the Legendre symbol $\left(\frac{d}{p_i}\right)\neq1$ (it is well-known that there are infinitely many of these).
\begin{lemma}\label{equality-in-quadratic-case}
For any prime $p,$ if $\left(\frac{d}{p}\right)=1,$ then any element of $R_+(\sqrt{d})$ can be multiplied by an integer not divisible by $p$ to obtain an element of $\OO_{\Q(\sqrt{d})}.$
\end{lemma}
\begin{proof}
It suffices to show that $(\sqrt{d})_p\in\OO_{\Q(\sqrt{d})}.$ Now, note that $x^2-d$ splits in $\Z_p[x]$ with two roots $z_1,z_2$ satisfying $z_1=-z_2,$ and so we apply Definition \ref{a_p} to see that $N=0$ and $K=k_1=k_2=v_p(z_1-z_2)=k-1,$ so $m_1=m_2=k.$ Thus, we have $(\sqrt{d})_p=\frac{(\sqrt{d}-r_1)(\sqrt{d}-r_2)}{p^{2k-1}}.$ Note that we can take $r_1=-r_2$ to reduce $$
(\sqrt{d})_p=\frac{d-r_1^2}{p^{2k-1}}\in\Z\subset\OO_{\Q(\sqrt{d})},
$$
because if $p\neq2$ then $k=1$ so $p^{2k-1}=p$ and $d-r_1^2$ is divisible by $p,$ and if $p=2$ then $k=2$ so $p^{2k-1}=2^3=8,$ but $d\equiv1\pmod8$ and every odd number squared is equivalent to $1\pmod8,$ so $d-r_1^2$ is divisible by $8.$ This completes the proof.
\end{proof}
Now, for every squarefree integer $d,$ define $\PP_d$ as the set of primes $p$ such that $\left(\frac{d}{p}\right)\neq1.$ Then we obtain the following characterization.
\begin{theorem}\label{characterization-when-sqrt}
For any squarefree integer $d,$ the ring $R_+(\sqrt{d})=\OO_{\Q(\sqrt{d})}\left[\frac{1}{p_1},\frac{1}{p_2},\dots\right],$ where $p_i$ are the primes in $\PP_d.$
\end{theorem}
\begin{proof}
This is a simple application of Theorem \ref{final-description}. If $p\in\PP_d,$ then by Lemma \ref{contains-inverses-of-primes}, $(\sqrt{d})_p=\frac{1}{p}.$ If $p\not\in\PP_d,$ then by Lemma \ref{equality-in-quadratic-case}, $(\sqrt{d})_p\in\OO_{\Q(\sqrt{d})}.$ The result follows.
\end{proof}
\subsection{Roots of Unity}
A root of unity is a number $x$ such that $x^n=1.$ A primitive $n$'th root of unity is a number $x$ such that $x^n=1$ and $x^k\neq1$ for any $k<n.$ For each positive integer $n,$ there are exactly $\varphi(n)$ primitive $n$'th roots of unity, where $\varphi$ stands for Euler's totient function. Specifically, say $\zeta$ is one such root of unity. Then $\zeta^k,$ where $k$ is relatively prime to $n,$ are all the primitive $n$'th roots of unity. These are conjugates of each other, with minimal polynomial $\Phi_n(x)\mid(x^n-1).$ Now, roots of unity are algebraic integers, in particular the discriminant of $\Phi_n(x)$ divides the discriminant of $x^n-1.$ We can calculate the discriminant of $x^n-1$ explicitly by the formula \begin{align*}
    D(x^n-1)&=\prod_{1\leq i<j\leq n}(\zeta^i-\zeta^j)^2=(-1)^{n(n-1)/2}\prod_{1\leq i\leq n}\prod_{1\leq j\leq n,j\neq i}(\zeta^i-\zeta^j)\\
    &=(-1)^{n(n-1)/2}\prod_{1\leq i\leq n}\frac{d}{dx}(x^n-1)(\zeta^i)=(-1)^{n(n-1)/2}\prod_{1\leq i\leq n}n(\zeta^i)^{n-1}\\
    &=n^n(-1)^{n(n-1)/2}\zeta^{(n^3-n)/2}=n^n(-1)^{n(n+1)/2+1}.
\end{align*}
In particular, Lemma \ref{primes-not-dividing-discriminant} allows for a simpler calculation for all primes not dividing $n.$
\begin{lemma}
If $n>2,$ then for any prime $p$ that divides $n,$ the element $\frac{1}{p}\in R_+(\zeta).$
\end{lemma}
\begin{proof}
First, assume $p$ is an odd prime. It is sufficient to take a $p$'th root of unity $\tau=\zeta^{n/p}$ and prove that $\frac{1}{p}\in R_+(\tau)\subset R_+(\zeta).$ Now, we have $\frac{\tau^p-\tau}{p}=\frac{1-\tau}{p}\in R_+(\tau)$ (because $R_+(\tau)$ is a ring and the element $\tau^p-\tau$ is equivalent to $\tau(\tau-1)\cdots(\tau-p+1)$ modulo $p).$ Now, for any $0<i<p,$ we have $\frac{1-\tau^i}{p}\in R_+(\tau^i)=R_+(\tau)$ by the same reasoning. We multiply these $p-1$ elements to get $\frac{p-1}{p^{p-1}}\in R_+(\tau),$ so $\frac{p-1}{p}\in R_+(\tau).$ We subtract this element from $1$ to get that $\frac{1}{p}\in R_+(\tau),$ as desired.

Now, suppose $p=2.$ Then, since $n>2,$ either $4\mid n$ or $p\mid n$ for some odd prime $p.$ If $4\mid n,$ then the $4$th root of unity $i\in R_+(\zeta).$ In particular, we have $R_+(i)\subset R_+(\zeta).$ Reviewing the quadratic case, we see that $\frac{1}{2}\in R_+(i),$ as desired. If $p\mid n,$ then we consider once again $\tau$ a $p$'th root of unity. We have $\frac{\tau^2-\tau}{2}\in R_+(\zeta).$ We multiply by $\tau^{p-1}$ and multiply by $-1$ to get $\frac{1-\tau}{2}\in R_+(\zeta).$ Once again, we get $\frac{1-\tau^i}{2}\in R_+(\zeta)$ for $0<i<p.$ We multiply these elements to get $\frac{p-1}{2^{p-1}}\in R_+(\zeta).$ The denominator contains strictly more factors of $2$ than the numerator (since $2^x>x$ for all $x),$ so we have $\frac{1}{2}\in R_+(\zeta),$ as desired.
\end{proof}
This lemma allows us to entirely describe the structure of $R_+(\zeta).$
\begin{prop}
For any primitive $n$'th root of unity $\zeta,$ where $n>2,$ the ring $$
R_+(\zeta)=\Z\left[\zeta,\frac{1}{p_1},\frac{1}{p_2},\dots\right],
$$
where $p_i$ are the primes that are not equivalent to $1$ modulo $n.$
\end{prop}
\begin{proof}
For any prime $p$ not dividing $n,$ consider the equation $\Phi_n(x)\equiv0\pmod p.$ Note that, since $\Phi_n$ divides $x^n-1,$ we must have that any root $r$ is an $n$'th root of unity modulo $p.$ On the other hand, $\Phi_n$ divides $\frac{x^n-1}{x^k-1}$ for any $k$ dividing $n,$ and the discriminant of $x^n-1$ is nonzero, hence $r$ is not a $k$'th root of unity modulo $p$ for any $k<n.$ Therefore, $r$ is a primitive $n$'th root of unity modulo $p.$ Note that $r$ is a member of the group of units of $\Z/p\Z,$ equal to $(\Z/p\Z)^\times\simeq\Z/(p-1)\Z.$ Since $r$ has order $n,$ by Lagrange's Theorem $n$ must divide the order $p-1$ of the group for such a root $r$ to exist. Thus, if $p\not\equiv1\pmod n,$ then there are no roots of $\Phi_n$ modulo $p,$ so $\frac{1}{p}\in R_+(\zeta).$ Now, suppose $p\equiv1\pmod n,$ in particular $p\nmid n,$ so $p$ does not divide the discriminant of $\Phi_n.$ Note that $$
\Phi_n\mid(x^n-1)\mid(x^{p-1}-1)\mid(x^p-x),
$$
and $x^p-x$ splits entirely modulo $p.$ Hence, $\Phi_n$ splits modulo $p,$ so we have $$
\zeta_p=\frac{\zeta(\zeta-1)\cdots(\zeta-p+1)}{p}\equiv\frac{\zeta^p-\zeta}{p}=0\pmod{1},
$$
so $\zeta_p\in\Z[\zeta].$ Finally, it is possible to show that the ring $\Z[\zeta]$ is integrally closed, i.e. it is its own ring of integers \cite{milneANT}. Applying Theorem \ref{final-description}, we obtain the desired result.  
\end{proof}

\appendix





\section{Constructive proofs of some results of section 7}
\begin{proof}[Proof of Lemma \ref{equality-in-quadratic-case}]
We demonstrate an explicit construction for $p=2.$ The case where $p$ is an odd prime is very similar and follows much more immediately from previous theorems (as no repeated roots$\pmod p$ may extend for $p\neq2).$

Say $d$ is an odd number with $d\equiv x_k^2\pmod{2^k}$ for some integer $x_k$ for all integers $k.$ Thus, $d\equiv1\pmod8$ (this follows from an examination of the multiplicative group of units of $\Z/2^k\Z, $ isomorphic to $\Z/2\Z\times\Z/2^{k-2}\Z)$ and it has four square roots modulo $2^k$ for any $k\geq3,$ call them $x_k,y_k,2^k-x_k,2^k-y_k.$ Now, we prove by induction that, for any integers $l>k,$ $$
(\sqrt{d}+k)\dots(\sqrt{d}+l)=2^{v_2((l-k+1)!)-1}(a+b\sqrt{d}),
$$
with $a,b\in\Z.$

The base case is $l-k+1\leq8,$ for which we can explicitly calculate the product to see that it is divisible by the necessary number of $2$'s.

Now, suppose that for any $k,l$ with $l-k+1\leq2^n$ (including $n=3),$ that $$
(\sqrt{d}+k)\dots(\sqrt{d}+l)=2^{v_2((l-k+1)!)-1}(a+b\sqrt{d}),
$$
with $a,b\in\Z.$ Take any $l>k$ with $2^n<l-k+1\leq2^{n+1}.$ Then we have, by the inductive hypothesis, that $$
(\sqrt{d}+k)\dots(\sqrt{d}+l-2^n)=2^{v_2((l-k-2^n+1)!)-1}(a+b\sqrt{d})
$$
and $$
(\sqrt{d}+l-2^n+1)\dots(\sqrt{d}+l)=2^{v_2((2^n)!)-1}(c+e\sqrt{d}).
$$
First, assume that $l-k+1\neq2^{n+1}.$ Then, by Legendre's formula, as before, we have $$
v_2((l-k+1)!)=v_2((2^n)!)+v_2((l-k-2^n+1)!).
$$
We now show, by a calculation of norms, that both $a+b\sqrt{d}$ and $c+e\sqrt{d}$ have norm divisible by $4.$ To see this, consider the norms of the left sides, $(k^2-d)\dots((l-2^n)^2-d)$ and $((l-2^n+1)^2-d)\dots(l^2-d).$ In the sets $\{k,\dots,l-2^n\}$ and $\{l-2^n+1,\dots,l\},$ respectively, there are $l-k-2^n+1$ and $2^n$ consecutive numbers, among which at least $\floor{\frac{l-k-2^n+1}{2}}$ and $2^n,$ respectively, are square roots of $d\pmod8$ - all the numbers equivalent to $1\pmod2,$ because they are either $1,3,5$ or $7\pmod8,$ all of which square to $1\pmod8$ - at least $\floor{\frac{l-k-2^n+1}{4}},\floor{\frac{2^n}{4}}$ are square roots of $d\pmod{16}$ (half of the odd numbers are $1\pmod{16}$ when squared, and the other half are $9\pmod{16}),$ at least $\floor{\frac{l-k-2^n+1}{8}},\floor{\frac{2^n}{8}}$ are square roots of $d\pmod{32},$ and so on. We thus get that \begin{align*}
&v_2(N((\sqrt{d}+k)\dots(\sqrt{d}+l-2^n)))\\
&\geq3\floor{\frac{l-k-2^n+1}{2}}+\floor{\frac{l-k-2^n+1}{4}}+\floor{\frac{l-k-2^n+1}{8}}+\dots\\
&=2\floor{\frac{l-k-2^n+1}{2}}+v_2((l-k-2^n+1)!)\geq2v_2((l-k-2^n+1)!)
\end{align*}
and similarly $$
v_2(N((\sqrt{d}+l-2^n+1)\dots(\sqrt{d}+l)))\geq2v_2((2^n)!).
$$
We compare to the norms of the right hand sides, $$
v_2(2^{2v_2((l-k-2^n+1)!)-2}N(a+b\sqrt{d}))=2v_2((l-k-2^n+1)!)-2+v_2(N(a+b\sqrt{d}))
$$
and $$
v_2(2^{2v_2((2^n)!)-2}N(c+e\sqrt{d}))=2v_2((2^n)!)-2+v_2(N(c+e\sqrt{d})),
$$
respectively. We then see that both $v_2(N(a+b\sqrt{d})$ and $v_2(N(c+e\sqrt{d})$ are at least $2,$ so the norms are divisible by $4,$ as desired.

Say the norm of $a+b\sqrt{d},$ equal to $a^2-b^2d,$ is divisible by $4.$ Then $a^2\equiv b^2d\pmod4$ so either they are both $0$ or $1.$ Thus, either $a$ and $b$ are both even or they are both odd. If $a,b$ are both even, then $a+b\sqrt{d}$ is divisible by $2$ so $(a+b\sqrt{d})(c+e\sqrt{d})=g+f\sqrt{d}$ is divisible by $2.$ The same happens if $c,e$ are both even. If $a,b,c,e$ are all odd, then $g+f\sqrt{d}=(ac+bed)+(ae+bc)\sqrt{d}$ has both $ac+bed$ and $ae+bc$ even because $d$ is odd, so $g+f\sqrt{d}$ is divisible by $2.$ In any case, $g+f\sqrt{d}=2(h+i\sqrt{d})$ for some integers $h,i.$

Thus, $$
(\sqrt{d}+k)\dots(\sqrt{d}+l)=2^{v_2((l-k+1)!)-2}(a+b\sqrt{d})(c+e\sqrt{d})=2^{v_2((l-k+1)!)-1}(h+i\sqrt{d}),
$$
as desired. 

If $l+k-1=2^{n+1},$ then the integers $\{k,k+1,\dots,l\}$ span all the residues $\pmod{2^{n+1}}.$ For each odd $m$ with $k\leq m\leq l$ that is nonzero modulo $2,$ consider the corresponding $j\neq m$ (note that $n+1>1$ so $2^{n+1}\neq2)$ with $m+j\equiv0\pmod{2^{n+1}},$ and consider the product $$
(\sqrt{d}+m)(\sqrt{d}+j)=(d+mj)+(m+j)\sqrt{d}.
$$
If $m^2\equiv d\pmod{2^i}$ for some $i,$ then $$
d+mj\equiv d+m(p^{n+1}-m)\equiv d-m^2\equiv0\pmod{2^i}
$$
and $m+j\equiv0\pmod{2^i},$ so $(\sqrt{d}+m)(\sqrt{d}+j)$ is divisible by $2^i.$ This holds for any $i\leq n.$ Now, note that, from above, every odd integer is a square root of $d\pmod8,$ every fourth integer is a square root of $d\pmod{16},$ every eighth integer is a square root of $d\pmod{32},$ and so on, until we get to every $2^{n-2}$'th integer is a square root of $d\pmod{2^n},$ so we have, where $v_2(a+b\sqrt{d})=\min(v_2(a),v_2(b)),$ that \begin{align*}
v_2((\sqrt{d}+k)\dots(\sqrt{d}+l))&\geq3\floor{\frac{2^{n+1}}{2}}+\floor{\frac{2^{n+1}}{4}}+\floor{\frac{2^{n+1}}{8}}+\dots+\floor{\frac{2^{n+1}}{2^{n-2}}}\\
&=3\cdot2^n+2^{n-1}+\dots+8=2^{n+2}-8>2^{n+1}-1
\end{align*} for $n\geq2,$ which it is (in fact, $n\geq3).$ This completes the inductive step.

In particular, for any $k$ we have $$
\binom{\sqrt{d}}{k}=\frac{\sqrt{d}(\sqrt{d}-1)\dots(\sqrt{d-k+1})}{k!}=\frac{2^{v_2(k!)-1}(a+b\sqrt{d})}{k!}.
$$
Thus, all factors of $2$ but one cancel in the denominator. This immediately implies that $\frac{1}{2}\not\in R_+(\sqrt{d}),$ as otherwise $\left(\frac{1}{2}\right)^2\in R_+(\sqrt{d}).$ In this case, we see that $\frac{a+b\sqrt{d}}{2}\in R_+(\sqrt{d})$ only if $a,b$ are of the same sign, as if one is even, then we can subtract an integer or an integer multiple of $\sqrt{d}$ to get either $\frac{a}{2}\in R_+(\sqrt{d})$ or $\frac{b\sqrt{d}}{2}\in R_+(\sqrt{d}),$ respectively, implying that the other is also even (otherwise $\frac{1}{2}\in R_+(\sqrt{d})).$ This completes the proof.
\end{proof}
\begin{proof}[Proof of Lemma \ref{quadratic-ring}]
Here we explicitly demonstrate the inclusion of $-1\in R_+(\a),$ showing that it is a ring. For negative $d$ we immediately have $(\sqrt{d})^2=d\in R_+(\sqrt{d}).$ Since $d$ is a negative integer, $-d-1$ is a nonnegative integer. Hence, we have $d+(-d-1)=-1\in R_+(\sqrt{d}),$ as desired. If $d$ is positive, say $n<\sqrt{d}<n+1,$ then $$
\binom{\sqrt{d}}{n+2}=a+b\sqrt{d}<0,a,b\in\Q.
$$
This is negative because the numerator is $\sqrt{d}(\sqrt{d}-1)\cdots(\sqrt{d}-n-1),$ and all the factors in the product but $\sqrt{d}-n-1$ are positive. We multiply by some positive integer to clear the denominators of $a,b,$ to get $$
c+e\sqrt{d}<0\in R_+(\sqrt{d}),c,e\in\Z.
$$
Either $c>e\sqrt{d}$ or $c<e\sqrt{d}.$ Assume without loss of generality that $c>e\sqrt{d}$; the other case is analogous. Then $e$ must be negative (otherwise both $c$ and $e\sqrt{d}$ must be nonnegative, so their sum is as well), so we can add $-2e\sqrt{d}$ to get $c-e\sqrt{d}\in R_+(\sqrt{d}).$ If this element is positive, then we multiply $c+e\sqrt{d}$ by $c-e\sqrt{d}$ to get $c^2-e^2d<0\in R_+(\sqrt{d}),$ and then we repeat the argument in the previous paragraph to get that $-1\in R_+(\sqrt{d}).$ If this element is negative, then $c$ must be negative, so we can add $-2c$ to get $-c-e\sqrt{d}\in R_+(\sqrt{d}).$ Then we multiply $c-e\sqrt{d}<0$ by $-c-e\sqrt{d}>0$ to get $e^2d-c^2<0\in R_+(\sqrt{d}).$ Thus, the same argument applies, as desired.
\end{proof}
\begin{proof}[Proof of Lemma \ref{contains-inverses-of-primes}]
The connections made may seem too abstract - we verify the proof explicitly.

We first consider the case when $d$ is not divisible by $p$ and $p\neq2.$ Then $d$ is a quadratic nonresidue $\pmod p,$ so $x^2-d$ has no solutions modulo $p.$ Hence, by Theorem \ref{final-description}, we have $\a_p=(\sqrt{d})_p=\frac{1}{p}\in R_+(\sqrt{d}),$ as desired.

Now, suppose $p\mid d.$ Then, for any integers $a,b$ we have that the norm $N(a+b\sqrt{d})$ of $a+b\sqrt{d}$ is equal to $a^2-b^2d,$ which is not divisible by $p$ unless $a$ is. In particular, none of $\sqrt{d}-1,\sqrt{d}-2,\dots,\sqrt{d}-p+1$ are divisible by $p.$ Hence, we have $$
(p-1)!\binom{\sqrt{d}}{p}=\frac{\sqrt{d}(\sqrt{d}-1)\cdots(\sqrt{d}-p+1)}{p}\in R_+(\sqrt{d}),
$$
so by the norm argument, the element $\frac{k\sqrt{d}}{p}\in R_+(\sqrt{d}),$ for some integer $k$ not divisible by $p.$ But then we multiply by the multiplicative inverse of $k$ and subtract an integer multiple of $\sqrt{d}$ to obtain that $\frac{\sqrt{d}}{p}\in R_+(\sqrt{d}).$ Squaring this, we have $\frac{d}{p^2}\in R_+(\sqrt{d}),$ but $d$ is squarefree. Hence, $d$ is only divisible by a single copy of $p,$ implying that $\frac{a}{p}\in R_+(\sqrt{d})$ for $a=d/p$ not divisible by $p.$ Hence, $\frac{1}{p}\in R_+(\sqrt{d}),$ as desired.

Finally, suppose $p=2.$ We just need to consider the cases when $d\not\equiv1\pmod8,$ and the case when $d$ is divisible by $2$ has already been discussed above. First, suppose $d\equiv3\pmod4.$ Then $\binom{\sqrt{d}}{2}=\frac{d-\sqrt{d}}{2}\in R_+(\sqrt{d}).$ But we add $\sqrt{d}$ to get that $\frac{d+\sqrt{d}}{2}\in R_+(\sqrt{d})$ as well. Then we multiply these two to get that $\frac{d^2-d}{4}\in R_+(\sqrt{d}).$ Since $d\equiv3\pmod4,$ $d^2-d\equiv2\pmod4.$ Hence, we have $\frac{a}{2}\in R_+(\sqrt{d})$ for some odd number $a,$ so $\frac{1}{2}\in R_+(\sqrt{d}),$ as desired. Now, if $d\equiv5\pmod8,$ then, once again, we have that $\frac{d-\sqrt{d}}{2}\in R_+(\sqrt{d}).$ But then we have $$
\binom{(d-\sqrt{d})/2}{2}=\frac{(d-\sqrt{d})/2*(d-2-\sqrt{d})/2}{2}=\frac{d^2-d-2(d-1)\sqrt{d}}{8}\in R_+(\sqrt{d}).
$$
Note that $2(d-1)$ is divisible by $8,$ hence we can add $\frac{2(d-1)}{8}\sqrt{d}$ to get that $\frac{d^2-d}{8}\in R_+(\sqrt{d}).$ Then we repeat the argument for $d\equiv3\pmod4$ to obtain that $\frac{1}{2}\in R_+(\sqrt{d}),$ as desired.
\end{proof}
\unappendix
\bibliographystyle{alpha}

\bibliography{biblio}

\end{document}